\documentclass[11pt,twoside,reqno]{amsart}
\usepackage{amsmath, amsfonts, amsthm, amssymb,amscd, graphicx, amscd,cite}
\usepackage{float,epsf}
\usepackage[english]{babel}
\usepackage{enumerate}
\usepackage{tikz}

\usepackage{srcltx}
\usepackage{geometry}
\usepackage{verbatim}
\usepackage{mathrsfs}
\usepackage{hyperref}

\newcommand{\R}{\mathbb{R}}

 \newcommand{\GG}{\mathcal{G}}

\newcommand{\ve}{\varepsilon}

 \newcommand{\Lip}{\text{\rm Lip}}
\newcommand{\ban}[1]{\left\langle  #1 \right\rangle}
 
\newtheorem{theorem}{Theorem}[section]
\newtheorem{lemma}[theorem]{Lemma}

\newtheorem{definition}[theorem]{Definition}

\newtheorem{proposition}[theorem]{Proposition}
\newtheorem{corollary}[theorem]{Corollary}
\newtheorem{remark}[theorem]{Remark}
\newtheorem{example}[theorem]{Example}

\numberwithin{equation}{section}
\numberwithin{figure}{section}

\newcommand{\norm}[1]{\left\|#1\right\|}

\newcommand{\FF}{{\boldsymbol F}}
\renewcommand{\GG}{{\boldsymbol G}}

\newcommand{\tildef}[1]{\widetilde{ #1}}

\newcommand{\DM}{\mathcal D\mathcal M} %Divergence-measure space
\newcommand{\redb }{\partial^{*}} %reduced boundary
\renewcommand{\div}{\mathrm{div}} %divergence
\newcommand{\res}{\mathop{\hbox{\vrule height 7pt width .5pt depth 0pt
\vrule height .5pt width 6pt depth 0pt}}\nolimits}

\newcommand{\Haus}[1]{{\mathscr H}^{#1}} % Hausdorff measure
\newcommand{\Leb}[1]{{\mathscr L}^{#1}} % Lebesgue measure
\renewcommand{\div}{\text{\sl div}}
\newcommand{\dr}{{\rm d}}

\def\intave#1{\int_{#1}\hbox{\llap{$\raise2.3pt\hbox{\vrule
height.9pt width7pt}\phantom{\scriptstyle{#1}}\mkern-2mu$}}}

\begin{document}
\title[Traces and Extensions of Divergence-Measure Fields on Rough Open Sets]{Traces and Extensions \\ of Bounded Divergence-Measure Fields\\ on Rough Open Sets}

\author{Gui-Qiang Chen}
\address{Gui-Qiang G. Chen, Mathematical Institute, University of Oxford,
Oxford, OX2 6GG, UK}
\email{chengq@maths.ox.ac.uk}

\author{Qinfeng Li}
\address{Qinfeng Li,
Department of Mathematics, University of Texas at San Antonio,
San Antonio, TX 78249, USA}
\email{liqinfeng1989@gmail.com}

\author{Monica Torres}
\address{Monica Torres,
Department of Mathematics, Purdue University,
 West Lafayette, IN 47907-2067, USA}
 \email{torres@math.purdue.edu}

\dedicatory{{\it In memoriam} William P. Ziemer}

\begin{abstract}
We prove that an open set $\Omega \subset \mathbb{R}^n$
can be approximated by smooth sets of uniformly bounded perimeter from the interior
{\it if and only if} the open set $\Omega$ satisfies
\begin{align*}
&\qquad \qquad\qquad\qquad\qquad\qquad\qquad \mathscr{H}^{n-1}(\partial \Omega \setminus \Omega^0)<\infty,
\qquad &&\quad\qquad\qquad \qquad\qquad (*)
\end{align*}
where $\Omega^0$ is the measure-theoretic exterior of $\Omega$.
Furthermore, we show that condition (*) implies
that the open set $\Omega$ is an extension domain
for bounded divergence-measure fields, which improves the previous results
that require a strong condition that
$\mathscr{H}^{n-1}(\partial \Omega)<\infty$.
As an application, we establish a Gauss-Green formula {\it up to the boundary}
on any open set $\Omega$
satisfying condition (*) for bounded divergence-measure fields, for which
the corresponding normal trace is shown to be a bounded function concentrated on
$\partial \Omega \setminus \Omega^0$.
This new formula does not require the set of integration to be compactly contained
in the domain where the vector field is defined.
In addition, we also analyze the solvability of the divergence equation
on a rough domain with prescribed trace on the boundary, as well as
the extension domains for bounded $BV$ functions.
\end{abstract}
\maketitle

\section{Introduction}
We are concerned with necessary and sufficient conditions for open sets to be
approximated by smooth sets of uniformly bounded perimeter from their interior.
One of the motivations is from the crack problems in elasticity and materials science,
in which the domains under consideration often have lower dimensional cracks.
In this paper, we identify the following condition on a general open set $\Omega$ in the form:
\begin{align} \label{neat-2}
\mathscr{H}^{n-1}(\partial \Omega \setminus \Omega^0)<\infty
\end{align}
with $\Omega^0$ as the measure-theoretic exterior of $\Omega$,
and prove that \eqref{neat-2} is a necessary and sufficient condition
for the open set $\Omega$ to be approximated by smooth sets of uniformly bounded perimeter from
its interior.
Furthermore, we show that condition \eqref{neat-2} implies
that not only the open set $\Omega$ is an extension domain
for bounded divergence-measure fields,
but also
a Gauss-Green formula holds {\it up to the boundary}
of $\Omega$
for bounded divergence-measure fields, for which
the corresponding normal trace is shown to be a bounded function concentrated on
$\partial \Omega \setminus \Omega^0$.
This new formula does not require the set of integration to be compactly contained
in the domain where the vector field is defined.
Our formula takes into account cracks, since the integration takes place in $\Omega$,
instead of $\Omega^1$ (measure-theoretic interior) as in the previously established
formula \eqref{mia} below.

More precisely, the Gauss-Green formula is a fundamental formula in analysis in order to perform
integration by parts.
In the simplest form, it can be stated for a smooth vector field $\FF$ and a smooth bounded open set $E$
as follows:
\begin{equation}\label{basica}
\int_{E} \div \FF \, \dr x = - \int_{\partial E} \FF (y) \cdot \nu_{E}(y)\, \dr \mathscr{H}^{n-1}(y),
\end{equation}
where $\nu_E$ is the interior unit normal to $E$.
However, in many applications, it is necessary to integrate by parts for vector fields that are
only weakly differentiable and on domains with less regularity.
Then a fundamental question is how formula \eqref{basica} can be generalized to rough open sets and
weakly differentiable vector fields.
The first classical generalization is obtained by considering the left side of \eqref{basica}
as a linear functional acting on vector fields $\FF \in C_c^1(\mathbb{R}^n)$.
If $E$ is such that
\begin{equation}
\label{condicion}
\sup_{\FF \in C_c^1(\mathbb{R}^n)} \int_{E} \div \FF\, \dr x
< \infty,
\end{equation}
then the Riesz representation theorem implies that there exists a vector-valued Radon measure $\mu_E$ such that
\begin{equation}
\label{segunda}
\int_{E} \div \FF \, \dr x = \int_{\mathbb{R}^n} \FF \cdot \dr \mu_E \qquad  \text{  for all } \FF \in C_c^1(\mathbb{R}^n).
\end{equation}
A set $E$ that satisfies \eqref{condicion} is called {\it a set of finite perimeter} in $\mathbb{R}^n$.
The structure theorem of De Giorgi (see also \cite[Theorem 15.9]{Maggi} and \cite[\S\,5.5--\S\,5.7]{Ziemer})
shows that a set $E$ of finite perimeter has many regularity properties.
In particular, the topological boundary of $E$, denoted as $\partial E$,
contains an $(n-1)$--rectifiable set that is known as the reduced boundary of $E$, denoted as $\partial^* E$.
It can be shown that
every $x\in \partial^*E$ has an inner unit normal $\nu_{E}(x)$
and a tangent plane in the measure-theoretic sense (see \cite[Theorem 5.6.5]{Ziemer}).
Moreover, the Radon measure $\mu_E$ has the following structure:
\begin{equation}
 \mu_E = -\nu_E \mathscr{H}^{n-1} \res \partial^* E
\end{equation}
such that \eqref{segunda} reduces to
\begin{equation}\label{segunda2}
\int_E \div \FF \, \dr x = -\int_{\partial^* E} \FF (y) \cdot \nu_E (y) \, \dr \mathscr{H}^{n-1}(y).
\end{equation}
The Gauss-Green formula \eqref{segunda2} is also true for Lipschitz and $BV$ vector fields
over any set $E$ of finite perimeter 
(see De Giorgi \cite{degiorgi1961complementi, degiorgi1961frontiere}, Federer \cite{Federer1, Federer2}, and  Burago-Maz'ya \cite{BM2}).
The Gauss-Green formula for bounded vector fields $\FF$ with $\div \FF$ as a measure
on bounded Lipschitz domains
was studied by Anzellotti in \cite{A1983,A1983traces}.

In the context of hyperbolic conservation laws, the question of extending the Gauss-Green formula
to divergence-measure vector fields was first addressed
in Chen-Frid \cite{CF1,CF2}, as required in the analysis of
weak entropy solutions obeying the Lax entropy inequality.
In general, divergence-measure fields are vector fields $\FF \in L^p$, $1 \leq p \leq \infty$,
such that the distributional divergence $\div \FF$ is a measure, which
are much wider than $BV$ vector fields.
The theory of divergence-measure fields is distinct in nature depending on whether $ \FF \in L^{\infty}$, or $\FF \in L^p$
for $p \neq \infty$.
The Gauss-Green formula for divergence-measure fields in $L^{\infty}$ over sets of finite perimeter
was obtained
simultaneously and
independently by Chen-Torres \cite{ChenTorres} and  \v{S}ilhav\'y \cite{Silhavy1} in 2005 via different approaches.
The approach developed in \cite{ChenTorres} is to employ a product rule for bounded divergence-measure fields
to obtain a measure concentrated on $\partial^* E$ which is absolutely continuous with respect to $\mathscr{H}^{n-1}$
and whose density is the normal trace.
In \v{S}ilhav\'y \cite{Silhavy1}, a formula for the normal trace is provided,
which is defined for $\mathscr{H}^{n-1}$-{\it a.e.} $ x \in  \partial^* E$
as a limit of averages on balls $B(x,r)$.
Motivated by hyperbolic conservation laws and the insights in \cite{CF1,CF2},
the {\it interior} and {\it exterior} normal traces need to be constructed as the limit of classical normal traces
on one-sided smooth approximations of the set of finite perimeter, respectively,
so that the surface of a {\it shock wave} can be approximated with smooth surfaces to obtain the {\it interior} and {\it exterior}
fluxes on the {\it shock wave}.
In order to accomplish this, a new approximation theorem for sets of finite perimeter was developed
in Chen-Torres-Ziemer \cite{ctz} and Comi-Torres \cite{ComiTorres} which shows that the level sets of
convolutions  $w_k= \chi_{E} * \rho_{k}$ provide smooth approximations essentially
from the {\it interior} (by choosing $w_{k}^{-1}(t)$ for $\frac{1}{2} < t <1$)
and the {\it exterior} (for $0<t< \frac{1}{2}$).
In this approach, it is also assumed that $E \Subset \Omega$, where $\Omega$ is the domain of definition
of $\FF$, since the level set $w_{k}^{-1}(t)$ (with a suitable fixed $0<t < \frac{1}{2}$) can
intersect the measure-theoretic exterior $E^0$ of $E$.
A critical step in the proof is to show that $\mathscr{H}^{n-1} (w_{k}^{-1}(t) \cap E^{0}) \to 0$
as  $k \to \infty$, which is obtained by using the tangential properties of blow-ups
around the points on the reduced boundary of the sets of finite perimeter.

\smallskip
Given $\FF \in \DM^{p}(\Omega)$ for $1 \le p \le \infty$ and a bounded Borel set $E \subset \Omega$,
we can define the \textit{normal trace} of $\FF$ on $\partial E$ as the following distribution:
\begin{equation} \label{definitionofnormaltrace}
\ban{\FF \cdot \nu, \phi}_{\partial E} := \int_{E} \phi \, \dr \div \FF + \int_{E} \FF \cdot \nabla \phi \, \dr x
\qquad\,\, \mbox{for any $\phi \in \Lip_{c}(\R^{n})$}.
\end{equation}

If $\tildef{E}$ is any Borel set such that  $\Leb{n}((E \setminus \tilde{E}) \cup  (\tilde{E} \setminus E))=0$, then
$$
\ban{\FF \cdot \nu, \phi}_{\partial E} \neq \ban{\FF \cdot \nu, \phi}_{\partial \tilde{E}},
$$
unless $|\div \FF| \ll \Leb{n}$.
In particular, if $E \subset \Omega$ is an open set with smooth boundary,
then $\partial E = \partial \overline{E}$; however, if $|\div \FF|(\partial E) \neq 0$,
the normal traces of $\FF$
on the boundary of $E$ and $\overline{E}$ are different in general.
The normal trace is a distribution concentrated on the topological boundary $\partial E$ (see \cite[\S\,4]{CCT}).

If $\FF \in \DM^{\infty}_{\rm loc}(\Omega)$, and $E \Subset \Omega$ is a set of finite perimeter,
and if  $\tildef{E} = E^{1}$ or $\tildef{E} = E^{1} \cup \redb E$,
then
$|\ban{\FF \cdot \nu, \,\cdot}_{\partial \tildef{E}}| \ll \Haus{n - 1} \res \redb E$
with density in $L^{\infty}(\redb E; \Haus{n - 1})$.
Thus, for any $\phi \in \Lip_{c}(\Omega)$,
\begin{align}
&\ban{\FF \cdot \nu, \phi}_{\partial E^{1}}
 = - \int_{\redb E} \phi \, \mathfrak{F}_{\rm i} \cdot \nu_{E} \, \dr \Haus{n - 1}, \label{G-G phi Sobolev int}\\
&\ban{\FF \cdot \nu, \phi}_{\partial (E^{1} \cup \redb E)}
 = - \int_{\redb E} \phi \, \mathfrak{F}_{\rm e} \cdot \nu_{E} \, \dr \Haus{n - 1}, \label{G-G phi Sobolev ext}
\end{align}
where $\mathfrak{F}_{\rm i} \cdot \nu_{E},\, \mathfrak{F}_{\rm e} \cdot \nu_{E} \in L^{\infty}(\redb E; \Haus{n - 1})$
are the interior and exterior normal traces of $\FF$ as introduced in \cite[Theorem 5.3]{ctz}.

A Gauss-Green formula for divergence-measure fields $\FF \in L^p$, $p \neq \infty$,
and extended divergence-measure fields ({\it i.e.} $\FF$ is a vector-valued measure
whose divergence is a Radon measure) was first obtained in Chen-Frid \cite{CF2} for Lipschitz
domains.
In \v{S}ilhav\'y \cite{Silhavy2}, a Gauss-Green formula for extended divergence-measure fields
was proved to be held over general open sets, which is the most general result
presently available so far.  A formula for the normal trace distribution is given
in \cite[Theorem 2.4, (2.5)]{Silhavy2} and Chen-Frid \cite[Theorem 3.1, (3.2)]{CF2}
as the limit of averages over the neighborhoods of the boundary.

Motivated by hyperbolic conservation laws and the approach developed in
\cite{ctz}, the Gauss-Green formulas for divergence-measure  fields in $L^{p}$ for $p \neq \infty$ on general open sets was further studied in Chen-Comi-Torres \cite{CCT}.
One of the main objectives in \cite{CCT} is to represent the normal trace as the limit of classical
normal traces over smooth approximations of the domain.
Roughly speaking, the approach in  \cite{CCT} is to differentiate under the integral sign in the
formulas \cite[Theorem 2.4, (2.5)]{Silhavy2} and Chen-Frid \cite[Theorem 3.1, (3.2)]{CF2}
in order to represent the normal trace as the limit of boundary integrals ({\it i.e.}  integrals of the classical normal traces  $\FF \cdot  \nu$ over appropriate smooth approximations of the
domain).

The existence of Lipschitz deformable boundaries and the problem of characterizing vector fields in $L^p$, $p \neq \infty$,
so that the normal trace \eqref{definitionofnormaltrace} can be represented by a measure,
have also been studied in \cite{CCT}. It was first shown in \v{S}ilhav\'y  \cite[Example 2.5]{Silhavy2}
that, for $p \neq \infty$, the distribution defined in \eqref{definitionofnormaltrace} may
not be a measure.

\smallskip
One of the main purposes of this paper is to analyze bounded open sets of finite perimeter
satisfying \eqref{neat-2}
and to prove that any $\FF \in \DM^{\infty}(\Omega)$ has a normal trace that is an $L^{\infty}$ function concentrated
on $\partial \Omega \setminus \Omega^0$, without further assuming that $\FF$ is defined outside $\Omega$.
Since this trace is concentrated on $\partial \Omega \setminus \Omega^0$,
this situation is not included in \eqref{G-G phi Sobolev int}.

If $\Omega$ is a bounded open set satisfying \eqref{neat-2}, the normal trace as a distribution is:
\begin{equation}
\ban{\FF \cdot \nu, \phi}_{\partial \Omega} := \int_{\Omega} \phi \, \dr \div \FF + \int_{\Omega} \FF \cdot \nabla \phi \, \dr x
\qquad \mbox{for any $\phi \in \Lip_{c}(\R^{n})$}.
\end{equation}
A simple example of an open set satisfying \eqref{neat-2} in two-dimensions
is
$$
\Omega:= \{x\,:\, |x|<1,\,\, \mbox{$x_2 \neq 0$ when $x_1>0$} \}.
$$
Segment $L:=\{x\,:\,0<x_1<1,\, x_2=0\}$ could represent a fracture in the open set $\Omega$.
$L$ is a subset of $\Omega^1$,
which is the measure-theoretic interior of $\Omega$ as defined in \eqref{density} below;
however,
$L$ is also a part of the
topological boundary of $\Omega$.
Note that formula \eqref{G-G phi Sobolev int} {\it does not recognize}
the fracture since $\Omega^1= \{ x: |x|<1\}$, so that the integration
by parts happens in the whole open disk, and not in the desired domain $\Omega$,
thus losing the information on the {\it cracks}.
However, it is shown in Theorem \ref{da}
that the Gauss-Green formula still holds
on any open set satisfying \eqref{neat-2} (see  \eqref{nuevo} below),
and the corresponding normal trace is a bounded function that may have its support
on $\partial \Omega \cap \Omega^1$ ({\it i.e.} on the fracture).
Moreover, this new {\it up to the boundary} Gauss-Green formula does not require the domain of
integration to be compactly contained in the domain of $\FF$.
In particular, our results provide the Gauss-Green formulas on the domains with lower dimensional cracks,
as long as the {\it cracks} have finite $\mathscr{H}^{n-1}$--measure.

The proof of Theorem \ref{da} relies on Theorem \ref{brandnew}, which states that \eqref{neat-2}
is a necessary and sufficient condition on $\Omega$ so that it can be approximated by a sequence of smooth sets,
from the interior, of uniformly bounded perimeter.
The construction of the desired approximating sequence in Theorem \ref{brandnew}
is based on a fine covering of $\partial \Omega$.
Another tool in the proof of Theorem \ref{da} is Theorem \ref{ahax},
which shows that an open set satisfying \eqref{neat-2} is an extension domain
for bounded divergence-measure fields.
Theorem \ref{ahax} generalizes \cite[Theorem 8.5]{ctz} and \cite[Theorem 5.3]{comi2017locally}
where it is assumed that $\mathscr{H}^{n-1}(\partial \Omega) < \infty$.
A typical example of a set that satisfies \eqref{neat-2} with topological boundary containing a wild set of points of density zero
and $\mathcal{L}^n (\Omega^0 \cap \partial \Omega)>0$ can be found in
Barozzi-Gonzalez-Massari \cite{BGM09} and Li-Torres \cite[Theorem 8.5]{LT}.
Related problems involving the theory of divergence-measure fields have
been studied
in \cite{comi2017locally, degiovanni1999cauchy,  Frid1, Frid2, scheven2016bv, scheven2017anzellotti, scheven2016dual, Sch, S2, S3, Silhavy2}.

As a byproduct of our results, we show in Proposition \ref{bvtrace} that, for any set $\Omega$ of finite perimeter with
\begin{align}\label{specialcase}
\mathscr{H}^{n-1}\left(\partial \Omega \cap \Omega^1\right)=0
\end{align}
and $u \in BV(\Omega) \cap L^{\infty}(\Omega)$,
there exists $u^* \in L^{\infty}(\partial^* \Omega)$ (which is the \textit{trace} of $u$) such that,
for $\mathscr{H}^{n-1}$--{\it a.e.} $x \in \partial ^* \Omega$ and
any $\phi \in C_c^{1}(\mathbb{R}^n,\mathbb{R}^n)$,
the following integration by parts formula holds:
\begin{align*}
\label{cy4''}
    \int_{\Omega} \phi \cdot Du\, \dr x +\int_{\Omega} u\, \div \phi\, \dr x
    =-\int_{\partial^* \Omega} u^* \phi \cdot \nu_{\Omega}\, \dr\mathscr{H}^{n-1}
\end{align*}
with
\begin{align*}
\lim_{r \rightarrow 0} \frac{\int_{B_r(x) \cap \Omega}|u(y)-u^*(x)|\,\dr y}{r^n}=0  \qquad \, a.e. \,\,  x \in \partial^* \Omega.
\end{align*}
Proposition \ref{bvtrace} has applications to the shape optimization problem of the form:
\begin{equation}
\label{opt}
\inf_{u \in H^1(\Omega; \mathbb{S}^{n-1})} J(u,\Omega)
\end{equation}
with
\begin{equation}\label{opt-2}
J(u,\Omega):=\int_{\Omega} |\nabla u|^2\, \dr x+\int_{\partial^* \Omega} f\, u\cdot \nu_{\Omega} \, \dr \mathscr{H}^{n-1},
\end{equation}
where the minimization takes place over rough sets satisfying \eqref{specialcase} and $f$ is a given function that depends on the particular optimization problem.
The fact that the traces (up to the boundary) of bounded $BV$ functions can be defined on the reduced boundary
of $\Omega$ allows to show that the surface energy  \eqref{opt}--\eqref{opt-2},
involving the traces of bounded $H^1$ vector fields,
is well posed for the sets satisfying \eqref{specialcase} (see Li-Wang \cite{LW}).
In particular, Lipschitz domains and outward minimizing sets (see Definition \ref{pseudoconvex}) satisfy condition \eqref{specialcase}.

This paper is organized as follows:
In \S\,2, we introduce some notations and basic properties of divergence-measure fields.
In \S\,3, we prove that \eqref{neat-2} characterizes the sets that can be approximated
by the sets of uniformly bounded perimeter from the interior.
In \S\,4,  we prove that the sets satisfying \eqref{neat-2} are extension domains for bounded divergence-measure
fields.
We also show the weak convergence properties of the trace operator which will be used to establish Theorem \ref{da}.
In \S\,5, we prove our Gauss-Green formula up to the boundary on extension domains for bounded divergence-measure fields.
We also re-discovered the classical Gauss-Green formula, Theorem \ref{toy'}, obtained in \cite{ChenTorres,Silhavy1, ctz, ComiTorres}.
In \S\,6, using our previous results, we analyze the solvability of the equation: $\div \FF=0$ on rough domains,
with prescribed trace on $\partial \Omega$.
In \S\,7, we analyze extension domains for bounded $BV$ functions and show that \eqref{neat-2} is a sufficient (but not a necessary)
condition for $\Omega$ to be an extension domain for bounded $BV$ functions.

\section{Basic Notations and Properties of Divergence-Measure Fields}

In this section, we present some basic notations and properties of divergence-measure fields
for the subsequent development.

Given $E\subset\mathbb{R}^n$, the Lebesgue measure of $E$ is denoted as $\mathcal{L}^n(E)$ or $|E|$.
The set of points of density $\alpha$ of $E$ is defined as
\begin{equation}\label{density}
E^\alpha:=\{x \in \mathbb{R}^n\,:\, \lim_{r \rightarrow 0} \frac{|B_r(x) \cap E|}{|B_r(x)|}=\alpha\}.
\end{equation}
We also define
\begin{equation}
\partial^{m} E  := \mathbb{R}^{n} \setminus (E^{1} \cup E^0),
\end{equation}
which is the measure-theoretic boundary of $E$,

Note that $E^0= (\mathbb{R}^n \setminus E)^1$.
For the reduced boundary $\partial^* E$ of $E$,
$$
\partial^* E \subset \partial^m E \subset \partial E,
$$
where $\partial E$ is the topological boundary of $E$.
The perimeter of $E$ in $\Omega$ is denoted as
$$
P(E;\Omega):= \mathscr{H}^{n-1} (\partial^* E \cap \Omega).
$$
If $\Omega=\mathbb{R}^{n}$, we simply write $P(E)$.
For more details, see
\cite{afp,eg2,giusti1984minimal,Maggi,Ziemer}.

Throughout this paper, we
use $\rho_{\epsilon} \in C_c^{\infty}$ to denote the standard symmetric mollifier,
and $\omega_n$ to be the volume of the unit ball in $\mathbb{R}^n$.
For a Radon measure $\mu$, we use $|\mu|$ to denote its total variation.
We also use symbol $a\lesssim_{n} b$ to represent that $a \le C(n) b$,
where $C(n)$ is a constant depending only on $n$.
For simplicity of exposition, we assume in this paper that $\Omega$ is always a bounded set,
but most of our results can be generalized to unbounded sets.

\begin{definition}\label{pseudoconvex}
We say that $E$ is an outward minimizing set {\rm (}or pseudoconvex{\rm )} in $\mathbb{R}^n$ if  $P(F) \ge P(E)$ for any $F \supset E$.
\end{definition}

The outward minimizing sets, which are the sets of nonnegative variational mean curvature,
satisfy condition \eqref{specialcase}.
The outward minimizing sets are natural generalizations of convex sets.
For example, if $n=2$, a connected outward minimizing set is equivalent to a convex set; see \cite{FF}.
This class of sets may have very rough boundary: For example, for $n \ge 3$, an outward minimizing set
can have a boundary of positive Lebesgue measure, as shown in \cite{BGM09}.

\begin{definition}
Given an integrable vector field $\FF$ on the open set $\Omega$,
$\div \FF$ is a distribution acting on $C_c^{\infty}(\Omega)$ such that,
for any test function $\phi \in C_c^{\infty}(\Omega)$,
\begin{align}
\ban{\div \FF,\phi}:=-\int_{\Omega} \FF \cdot \nabla \phi \, \dr x.
\end{align}
We say that $\FF$ is an $L^p$ divergence-measure field in the open set $\Omega$ for $1\le p\le \infty$
if $\FF \in L^p(\Omega)$ and
\begin{align}\label{allthetime}
\sup \Big\{\int_{\Omega} \FF \cdot \nabla\phi\,\dr x\,: \, \phi \in C_c^1(\Omega), |\phi|\le1 \Big\} < \infty.
\end{align}
Condition \eqref{allthetime} implies that $\div \FF$ is a finite Radon measure in $\Omega$ {\rm (}{\it i.e.} $|\div \FF|(\Omega) < \infty${\rm )}
so that
\begin{align}
\ban{\div \FF,\phi}= \int_{\Omega} \phi \, \dr\div \FF =- \int_{\Omega} \FF \cdot \nabla \phi\, \dr x.
\end{align}
The Banach space $\mathcal{DM}^p(\Omega)$ consists of all $L^p$ divergence-measure fields on $\Omega$
for any $1\le p\le \infty$.
\end{definition}

A product rule between essentially bounded divergence-measure fields and scalar functions
of bounded variations was first proved in Chen-Frid \cite[Theorem 3.1]{CF1}
(also see  \cite{Frid1}).

\begin{theorem}[Chen-Frid \cite{CF1}] \label{productruleinfty}
Let $g \in BV(\Omega) \cap L^{\infty}(\Omega)$ and $\FF \in \DM^{\infty}(\Omega)$.
Then $g \FF \in \DM^{\infty}(\Omega)$ and
\begin{equation}\label{PRDMF}
\mathrm{div}(g \FF) = g^{*} \mathrm{div}\FF + \overline{\FF \cdot Dg}
\end{equation}
in the sense of Radon measures on $\Omega$,
where $g^{*}$ is the precise representative of $g$,
and $\overline{\FF \cdot Dg}$ is a Radon measure
which is the weak-star limit of $\FF \cdot \nabla g_{\delta}$
for some mollification $g_{\epsilon}:=g \ast \rho_{\epsilon}$
and is absolutely continuous with respect to $|Dg|$.
In addition,
$$
|\overline{\FF \cdot D g}| \le \|\FF\|_{L^{\infty}(\Omega; \R^{n})} |D g|.
$$
\end{theorem}

Let $\FF \in \DM^{p}(\Omega)$ for $1 \le p \le \infty$,
and let $g \in L^{\infty}(\Omega) \cap BV(\Omega)$.
If $p = \infty$, then
\begin{align}\label{cui1}
|\div \FF|+|\div(g \FF)| \ll \Haus{n - 1},
\end{align}
as observed first in \cite[Proposition 3.1]{CF1}.
If $p \in [\frac{n}{n-1}, \infty)$,
$$
 |\div \FF|(B) = |\div (g \FF)|(B) = 0
$$
for any Borel set $B$ with $\sigma$ finite $\Haus{n - p'}$ measure
(see \cite[Theorem 3.2]{Silhavy1}).

We will use the following approximation result,
whose proof is similar to the analogous result
for $BV$ functions (see \cite{CF1,CF2,giusti1984minimal}).

\begin{proposition}\label{approximation}
Let $\FF \in \mathcal{DM}^p(\Omega)$ for $1\le p\le \infty$ and a bounded domain $\Omega$.
Then there exists $\FF_j \in C^{\infty}(\Omega;\mathbb{R}^n)$ such that
\begin{align}
\lim_{j \rightarrow \infty}|\div \FF_j|(\Omega)=|\div \FF|(\Omega), \qquad
\lim_{j\rightarrow \infty}\norm{\FF_j-\FF}_{L^1(\Omega)}=0 \label{coushu-a}
\end{align}
with
\begin{align}\label{coushu-b}
\sup_{j}\|\FF_j\|_{L^p(\Omega)}<\infty.
\end{align}
\end{proposition}

\begin{remark}\label{i2}
In fact, we can choose $\FF_j$ with the additional property
that $\|\FF_j\|_{L^p} \le \|\FF\|_{L^p}$ for all $1 \le p \le \infty$ in Proposition {\rm \ref{approximation}}.
This follows from the proof of Proposition {\rm \ref{approximation}} as in {\rm \cite{CF1,CF2}}
and the Young inequality for the convolutions.
\end{remark}

We will frequently apply the following two theorems due to Federer ({\it cf}. \cite{Federer}), which can also be found in
\cite{afp,eg2,Maggi,Ziemer}.

\begin{theorem}\label{Federertheorem}
If $E$ is a set of finite perimeter in $\mathbb{R}^{n}$, then $\mathbb{R}^n=E^1 \cup E^0 \cup \partial^m E$
and $\mathscr{H}^{n-1}(\partial^m E \setminus \partial^*E)=0$.
\end{theorem}

\begin{theorem}[Criteria for sets of finite perimeter]\label{criteria'}
If $E\subset \mathbb{R}^n$ satisfies $\mathscr{H}^{n-1}(\partial^mE)<\infty$, then $E$ is a set of finite perimeter.
\end{theorem}

Next, we introduce the definition of normal traces:
\begin{definition}\label{trac}
Given $\FF \in \DM^{p}(\Omega)$ for $1 \le p \le \infty$, and a bounded Borel set $E \subset \Omega$,
then the \textit{normal trace} of $\FF$ on $\partial E$ is defined as
\begin{equation} \label{normal trace def}
\ban{\FF \cdot \nu, \phi}_{\partial E} := \int_{E} \phi \, \dr \div \FF + \int_{E} \FF \cdot \nabla \phi \, \dr x
\qquad \mbox{for any $\phi \in \Lip_{c}(\R^{n})$}.
\end{equation}
\end{definition}

The following theorem was proved in
\cite{ChenTorres,Silhavy1} (see also \cite{ctz,ComiTorres}):
\begin{theorem} \label{toy'}
If $E \Subset \Omega$ is a set of finite perimeter,
and if $\FF \in \mathcal{DM}^{\infty}(\Omega)$,
then there exists $\mathfrak{F}_{\rm i} \cdot \nu_{E} \in L^{\infty}(\partial^* E; \mathscr{H}^{n-1})$ such that
\begin{align}\label{mia}
\int_{E^1} \, \dr \div \FF
=-2 \overline{\chi_E \FF \cdot D\chi_E}(\partial^* E)=-\int_{\partial^* E}\mathfrak{F}_{\rm i} \cdot \nu_{E}\,\dr\mathscr{H}^{n-1},
\end{align}
where $\overline{\chi_E \FF\cdot D\chi_E}$ is the weak-star limit of measures $\chi_E \FF \cdot \nabla (\chi_E * \rho_{\epsilon})$.
\end{theorem}

\section{Interior Approximation by sets of uniformly bounded perimeter}

In this section, we prove that \eqref{neat-2} characterizes the sets that can be approximated
by sets of uniformly bounded perimeter from the interior.

\begin{theorem}\label{brandnew}
Let $\Omega$ be a bounded set with $|\Omega|>0$.
Then there exists smooth sets $E_k \Subset \Omega$ such that $E_k \rightarrow \Omega$ in $L^1$ and  $\sup_k P(E_k) < \infty$
if and only if \eqref{neat-2} is satisfied.
\end{theorem}

\begin{proof} We divide the proof into two steps.

\smallskip
1. We first show the {\it ``only if''} part.
Let $E_i$ be the assumed approximating sequence. Then, by the lower semicontinuity  of $P(\cdot)$ (see \cite[Theorem 5.2.1]{Ziemer}),
we know that $\Omega$ is of finite perimeter.
In order to obtain \eqref{neat-2}, by Theorem \ref{Federertheorem}, it suffices to show
\begin{align}\label{neat'}
\mathscr{H}^{n-1}(\partial \Omega \cap \Omega^1)<\infty,
\end{align}
since $\partial\Omega\setminus \Omega^0=(\partial\Omega\cap \Omega^1)\cup \partial^m\Omega$.

\smallskip
Since $E_i \Subset\Omega$ for each $i$, by definition of the measure-theoretic interior, we have
$$
\lim_{r \rightarrow 0} \frac{|B_r(x) \cap (\Omega \setminus E_i)|}{\omega_nr^n}=1 \qquad
\, \mbox{for all $x \in \partial \Omega \cap \Omega^1$}.
$$
Therefore, for any $x  \in \partial \Omega \cap \Omega^1$, we can choose $0<r<\infty$ such that
\begin{align*}
\frac{|B_r(x) \cap (\Omega \setminus E_i)|}{\omega_nr^n}=\frac{1}{2}.
\end{align*}
From the relative isoperimetric inequality (see {\it e.g.} \cite[Remark 12.38]{Maggi}), we have
\begin{align}\label{sq}
P \left(\Omega \setminus E_i;B_r(x)\right)
&\,\ge c(n) \min\big\{|B_r(x) \cap (\Omega \setminus E_i)|^{\frac{n-1}{n}}, |B_r(x) \setminus (\Omega \setminus E_i)|^{\frac{n-1}{n}}\big\}\nonumber\\[2mm]
&\, = c_1(n) r^{n-1}.
\end{align}
Thus, by Vitali's covering theorem, we can find a family of countable disjoint balls $B_{r_j}(x_j)$ such that
\begin{align}
&\partial \Omega \cap \Omega^1 \subset \cup_j B_{5r_j}(x_j), \label{61}\\[1mm]
&\frac{|B_{r_j}(x_j) \cap (\Omega \setminus E_i)|}{\omega_n r_j^n}=\frac{1}{2},\label{62}\\[1mm]
&r_j^{n-1} \lesssim_n P\left(\Omega \setminus E_i;B_{r_j}(x_j)\right).\label{63}
\end{align}
Let $\delta_i=\sup_{j} r_j$. By \eqref{61}--\eqref{63}, we have
\begin{align}
\mathscr{H}_{5\delta_i}^{n-1}(\partial \Omega \cap \Omega^1)
\le &\, n\omega_n 5^{n-1} \sum_j r_j^{n-1}\nonumber\\
\lesssim_n &\, \sum_j P\left(\Omega \setminus E_i;B_{r_j}(x_j)\right)\nonumber \\
\le &\, P(\Omega \setminus E_i)\nonumber\\[1mm]
=&\, P(\Omega)+P(E_i),
\end{align}
where we have used that $\{B_{r_j}(x_j)\}_j$ are disjoint for the last inequality.

By \eqref{62}, we have
$$
\limsup_{i \rightarrow \infty} \delta_i\lesssim_n
\big(\frac{2}{\omega_n}\big)^{\frac{1}{n}}
\limsup_{i \rightarrow \infty} |\Omega\setminus E_i|^{\frac{1}{n}}=0.
$$
Therefore, by definition of the Hausdorff measure and from the discussion above, we have
\begin{align}
\mathscr{H}^{n-1}(\partial \Omega \cap \Omega^1)
=\lim_{i \rightarrow \infty} \mathscr{H}_{5\delta_i}^{n-1}(\partial \Omega \cap \Omega^1)
\lesssim_n P(\Omega)+ \limsup_{i \rightarrow \infty} P(E_i)
< \infty.
\end{align}

\medskip
2. Now we show the {\it ``if''} part. By Theorem \ref{criteria'}, under this assumption, $\Omega$ is of finite perimeter.

For any $\delta>0$ and $x \in \partial \Omega \cap \Omega^0$, by definition of the measure-theoretic exterior,
we can choose $0<r<\delta$ such that
\begin{align}\label{600}
\frac{|\Omega \cap B_r(x)|}{|B_r(x)|}<\frac{1}{2}.
\end{align}
By the relative isoperimetric inequality (see {\it e.g.} \cite[Proposition 12.37]{Maggi}),
there is a constant $c(n)$ such that
\begin{align}\label{601}
|\Omega \cap B_r(x)|^{\frac{n-1}{n}} \le c(n) P\left(\Omega; B_r(x)\right).
\end{align}
From the coarea formula, it follows that there exists a constant $r$
such that $\mathscr{H}^{n-1}(\partial B_r(x) \cap \partial^m \Omega)=0$,
while \eqref{600}--\eqref{601} still hold.
Therefore, applying the classical Gauss-Green formula \eqref{segunda2} to the vector field $\FF(y)=y-x$
on the set of finite perimeter $\Omega \cap B_r(x)$ and using \cite[Theorem 16.3]{Maggi},
we have
\begin{align}
n|\Omega \cap B_r(x)|=&\,\int_{\Omega \cap B_r(x)} \div_y (y-x)\,\dr y \nonumber \\
=&\, -\int_{\Omega^1 \cap \partial B_r(x)} (y-x) \cdot \nu_{B_r(x)}(y)\,\dr \mathscr{H}^{n-1}
  -\int_{B_r(x) \cap \partial^* \Omega} (y-x) \cdot \nu_{\Omega}(y)\,\dr \mathscr{H}^{n-1} \nonumber \\
\ge &\, r\mathscr{H}^{n-1}(\Omega^1 \cap \partial B_r(x))-rP\left(\Omega;B_r(x)\right).
\end{align}
This implies
\begin{align}\label{602}
r\mathscr{H}^{n-1}(\Omega^1 \cap \partial B_r(x)) \le n|\Omega \cap B_r(x)| + rP\left(\Omega;B_r(x)\right).
\end{align}
Moreover, it is clear that
\begin{align}
\label{603}
\frac{|\Omega \cap B_r(x)|^{\frac{1}{n}}}{r} \le \omega_n^{\frac{1}{n}}.
\end{align}
Combining \eqref{601} with \eqref{602}--\eqref{603}, we have
\begin{align}
\mathscr{H}^{n-1}(\Omega^1 \cap \partial B_r(x))
\le &\, \frac{n|\Omega \cap B_r(x)|}{r}+P(\Omega;B_r(x))\nonumber \\
= &\, n  |\Omega \cap B_r(x)|^{\frac{n-1}{n}}  \frac{|\Omega \cap B_r(x)|^{\frac{1}{n}}}{r} + P\left(\Omega;B_r(x)\right) \nonumber\\
\le &\, nc(n)\omega_n^{\frac{1}{n}} P\left(\Omega;B_r(x)\right)+P\left(\Omega;B_r(x)\right),
\end{align}
that is,
\begin{align}
\label{key61}
\mathscr{H}^{n-1}(\Omega^1 \cap \partial B_r(x)) \lesssim_n P\left(\Omega;B_r(x)\right).
\end{align}
From Besicovitch's covering theorem (see {\it e.g.} \cite[Theorem 5.1]{Maggi}),
it follows that there exist $\mathcal{F}_i$, $i=1,2, \cdots, \xi(n)$, so that
each family $\mathcal{F}_i$ contains countably disjoint balls with radius less than $\delta$
satisfying
$$
\partial \Omega \cap \Omega^0 \subset \cup_{i=1}^{\xi(n)} \cup_{B \in \mathcal{F}_i}B
$$
and, for each $B_r(x) \in \cup_{i=1}^{\xi(n)}\mathcal{F}_i$,
\eqref{key61} holds.
Our assumption \eqref{neat-2} implies the existence of a family $\mathcal{F}_0$ of balls such that
\begin{align}
&\sup_{B \in \mathcal{F}_0} \text{diam} (B) \le 2\delta,\label{3.22a}\\
&\partial \Omega \setminus \Omega^0 \subset \cup_{B \in \mathcal{F}_0} B,\label{3.22b}\\
&\sum_{B \in \mathcal{F}_0}\mathscr{H}^{n-1}(\partial B) \lesssim_n  \mathscr{H}^{n-1}(\partial \Omega \setminus \Omega^0 ).\label{606}
\end{align}
We may also require that, for any $B_r(x) \in \mathcal{F}_0$,
\begin{align}\label{605}
\mathscr{H}^{n-1}(\partial B_r(x) \cap \partial^* \Omega)=0.
\end{align}
Since there are countably many balls in $\cup_{i=0}^{\xi(n)}\mathcal{F}_i$,
we can assume that \eqref{605} holds for any $B_r(x) \in \cup_{i=0}^{\xi(n)}\mathcal{F}_i$.

Since $\partial \Omega$ is a compact set,
we may find finite balls $\{B_{r_k}(z_k)\}_{k=1}^N \subset \cup_{i=0}^{\xi(n)}\mathcal{F}_i$
covering $\partial \Omega$.
Let $E=\Omega \setminus \cup_{k=1}^N B_{r_k}(z_k)$ so that $E \Subset\Omega$.
We estimate
\begin{align*}
P(E)=&\,P\left(\Omega \setminus \cup_{k=1}^N B_{r_k}(z_k)\right)\\
=&\, P\left(\cup_{k=1}^N B_{r_k}(z_k); \Omega^1\right) + P\left(\Omega; \mathbb{R}^n \setminus \cup_{k=1}^N B_{r_k}(z_k)\right)\\
=&\, P\left(\cup_{k=1}^N B_{r_k}(z_k); \Omega^1\right)\\
\le &\, \sum_{k=1}^N P(B_{r_k}(z_k) ; \Omega^1)\\
\le &\, \sum_{i=1}^{\xi(n)} \sum_{B \in \mathcal{F}_i} \mathscr{H}^{n-1}(\partial B \cap \Omega^1) + \sum_{B \in \mathcal{F}_0} \mathscr{H}^{n-1}(\partial B \cap \Omega^1),
\end{align*}
where we have used  \eqref{605}, \cite[Theorem 16.3, (16.11)]{Maggi}, and $\partial \Omega \subset \cup_{k=1}^{N} B_{r_k}(z_k)$.
Using \eqref{key61}, \eqref{606}, and the fact that  the balls in $\mathcal{F}_i$ are disjoint for $1 \le i \le \xi(n)$,
we have
\begin{align*}
P(E) \lesssim_n &\, \sum_{i=1}^{\xi(n)} \sum_{B \in \mathcal{F}_i} P(\Omega; B) + \mathscr{H}^{n-1}(\partial \Omega \setminus \Omega^0)\\
\lesssim_n &\, \xi(n) P(\Omega) +  \mathscr{H}^{n-1}(\partial \Omega \setminus \Omega^0)\\
\lesssim_n &\,   \mathscr{H}^{n-1}(\partial \Omega \setminus \Omega^0).
\end{align*}
Since $0<r<\delta$ for any $B_r(x)$ in the cover of $\partial \Omega$, we can estimate
\begin{align*}
|\Omega \setminus E| \le & \,\sum_{i=0}^{\xi(n)} \sum_{B \in \mathcal{F}_i } |B \cap \Omega| \\
\lesssim_n & \,\sum_{i=1}^{\xi(n)} \sum_{B \in \mathcal{F}_i } |B \cap \Omega|^{\frac{1}{n}} |B \cap \Omega|^{\frac{n-1}{n}}
  + \delta \sum_{B \in \mathcal{F}_0 }\mathscr{H}^{n-1}(\partial B)\\
\lesssim_n &\, \delta \bigg(\sum_{i=1}^{\xi(n)} \sum_{B \in \mathcal{F}_i }P(\Omega;B)+\mathscr{H}^{n-1}(\partial \Omega\setminus \Omega^0)\bigg)\\
\lesssim_n &\, \delta\,\big(\xi(n) P(\Omega) +\mathscr{H}^{n-1}(\partial \Omega\setminus \Omega^0)\big)\\[1mm]
\lesssim_n &\, \delta \, \mathscr{H}^{n-1}(\partial \Omega\setminus \Omega^0),
\end{align*}
where we have used
\eqref{601}, \eqref{603}, \eqref{606}, and the fact that the balls in $\mathcal{F}_i$ are disjoint for $1\le i \le \xi(n)$.
Since $|\Omega|>0$, the previous construction shows that, for each $\delta > 0$ small, we can construct a set $E_{\delta} \ne \emptyset$ such that
\begin{align*}
    E_{\delta} \Subset\Omega, \qquad |\Omega \setminus E_{\delta}| \lesssim_n \delta \, \mathscr{H}^{n-1}(\partial \Omega\setminus \Omega^0),
\end{align*}
and
\begin{align}
\label{tiao'}
    P(E_{\delta}) \lesssim_n \mathscr{H}^{n-1}(\partial \Omega\setminus \Omega^0).
\end{align}
Choosing a sequence $\delta_k \to 0$ yields that Theorem \ref{brandnew} holds for sets $E_{\delta_k}$
with Lipschitz boundary.

Then we can employ the standard smoothing arguments ({\it e.g.} \cite{Maggi}) for sets $E_{\delta_k}$ to conclude that
smooth sets $E_k$ can be chosen  so that  Theorem \ref{brandnew} holds for $E_k$.
\end{proof}

Similarly, we have

\begin{corollary}\label{chidun}
Let $\Omega$ be a bounded set. Then there exist smooth sets $F_k \Supset \Omega$ such that $F_k \rightarrow \Omega$ in $L^1$ and  $\sup_k P(F_k) < \infty$
if and only if $$\mathscr{H}^{n-1}(\partial \Omega \setminus \Omega^1)<\infty.$$
\end{corollary}

As a direct consequence of Theorem \ref{brandnew} and Corollary \ref{chidun}, we have
\begin{corollary}
Let $\Omega$ be a bounded set with $|\Omega|>0$. Then
$$
\mathscr{H}^{n-1}(\partial \Omega)<\infty
$$
if and only if there exist sets $E_k \Subset \Omega \Subset F_k$ such that
\begin{align*}
&\sup_k P(E_k) + \sup_k P(F_k) <\infty,\\
&E_k \rightarrow \Omega,\quad  F_k \rightarrow \Omega \,\qquad \mbox{in $L^1$}.
\end{align*}
\end{corollary}

\begin{remark}
From the second part of the proof of Theorem {\rm \ref{brandnew}}, we have shown that,
if $\Omega$ satisfies \eqref{neat-2} with $|\Omega|>0$, then there exist smooth $E_k \Subset\Omega$
such that $E_k \rightarrow \Omega$ in $L^1$ and
\begin{align*}
   P(E_k) \lesssim_n \mathscr{H}^{n-1}(\partial \Omega \setminus \Omega^0).
\end{align*}
\end{remark}

\section{Extension Domains and Continuity of Traces for Bounded Divergence-measure Fields}

In this section, we prove that the sets satisfying \eqref{neat-2} are extension domains
for $L^\infty$ divergence-measure
fields. We also show the weak convergence properties of the trace operator which will be
used to establish Theorem \ref{da} in \S\,5.

\subsection{Extension domains}
Throughout the rest of the paper,
given $\FF \in \mathcal{DM}^{\infty} (\Omega)$, the extension of $\FF$ is
defined as
    \begin{equation}\label{j1}
    \tilde{\FF}(x):=
	\begin{cases}
	\FF(x) &\,\, \text{for $x \in \Omega$}, \\
    0 &\,\, \text{for $x \notin \Omega$}.
	\end{cases}
	\end{equation}

\begin{definition}
\label{Ext}
We say that $\Omega$ is an extension domain for bounded divergence-measure fields if, for any $\FF \in \mathcal{DM}^{\infty}(\Omega)$, $\tilde{\FF}$ is a divergence-measure field in
$\mathbb{R}^n${\rm ;} that is, \eqref{allthetime} holds with $\Omega=\mathbb{R}^n$ and
\begin{align}\label{ext}
    |\div \tilde{\FF}|(\mathbb{R}^n)<\infty.
\end{align}
\end{definition}

The next theorem extends \cite[Theorem 8.4]{ctz}.
\begin{theorem}
\label{ahax}
If $\Omega$ is a bounded open set satisfying \eqref{neat-2}, then $\Omega$ is an extension domain for bounded divergence-measure fields.
\end{theorem}

\begin{proof}
Let $\tilde{\FF}_{\epsilon}$ be the standard mollification of $\tilde{\FF}$.
Using the area formula, we see that, for any $x \in \mathbb{R}^n$,
$$
\lim_{\epsilon \rightarrow 0} \int_0^1 \int_{\partial B_r(x)}|\tilde{\FF}_{\epsilon}(y)-\tilde{\FF}(y)|\, \dr\mathscr{H}^{n-1}(y)\dr r
= \lim_{\epsilon \rightarrow 0}\int_{B_1(x)} |\tilde{\FF}_{\epsilon}(y)-\tilde{\FF}(y)|\, \dr y=0.
$$
From Fatou's lemma, there is a subsequence $\epsilon_j \rightarrow 0$ such that
\begin{align}\label{fol1}
\lim_{j\rightarrow \infty}\int_{\partial B_r(x)} |\tilde{\FF}_{\epsilon_j}(y)-\tilde{\FF}(y)|\, \dr\mathscr{H}^{n-1}(y)=0
\qquad\,\,\mbox{for {\it a.e.} $r>0$}.
\end{align}
Also, since $\div \tilde{\FF}=\div \FF$ is a finite measure on $\Omega$, we have
\begin{align}\label{fol2}
|\div \tilde{\FF}|(\partial B_r(x) \cap \Omega)=0 \qquad\,\,\mbox{for {\it a.e.} $r>0$.}
\end{align}
From Theorem \ref{brandnew}, there exist sets $E_k \Subset \Omega$ such that $E_k \rightarrow \Omega$ in $L^1$ and
$$
P(E_k) \lesssim_n \mathscr{H}^{n-1}(\partial \Omega \setminus \Omega^0).
$$
We note from the proof of Theorem \ref{brandnew} that each $E_k$ can be chosen with form $\Omega \setminus \cup_{i=1}^{N_k} B_{r_i^k}(x_i^k)$,
where $x_i^k \in \partial \Omega$.
As explained in the proof of Theorem \ref{brandnew}, $B_r(x) \in \{ B_{r_i^k}(x_i^k)\}_{1\le i \le N_k}$ can be
chosen so that \eqref{600} and \eqref{606}--\eqref{605} hold.
We can also choose $r$ such that \eqref{fol1}--\eqref{fol2} hold.
Since $\partial E_k \subset \cup_{i=1}^{N_k} \big(\partial B_{r_i^k}(x_i^k) \cap \Omega\big)$, we can choose $E_k \Subset \Omega$ such that
\begin{align}
&\lim_{j\rightarrow \infty}\int_{\partial E_k} |\tilde{\FF}_{\epsilon_j}(y)-\tilde{\FF}(y)|\dr\mathscr{H}^{n-1}(y)=0,\label{621}\\[1mm]
&|\div \tilde{\FF}|(\partial E_k)=0,\label{622}\\[1mm]
&E_k \rightarrow \Omega \, \,\,\, \mbox{in $L^1$},\\[1mm]
&P(E_k) \lesssim_n \mathscr{H}^{n-1}(\partial \Omega \setminus \Omega^0).
\end{align}
Applying the divergence theorem for smooth vector fields on sets of finite perimeter,
we know that, for any $\phi \in C_c^1(\mathbb{R}^n)$ with $|\phi| \le 1$,
\begin{align}\label{633}
\int_{E_k} \tilde{\FF}_{\epsilon_j} \cdot \nabla \phi \,\dr y
=- \int_{E_k} \phi \, \dr \div \tilde{\FF}_{\epsilon_j}-\int_{\partial^* E_k} \phi \tilde{\FF}_{\epsilon_j} \cdot \nu_{E_k}\, \dr \mathscr{H}^{n-1}.
\end{align}
Since $\tilde{\FF}_{\epsilon_j} \rightarrow \tilde{\FF}$ in $L^1$ as $j \rightarrow \infty$,
we have
$$
\div \tilde{\FF}_{\epsilon_j} \stackrel{*}{\rightharpoonup}  \div \tilde{\FF}.
$$
By \eqref{622}, for any $ \phi \in C_c^1(\mathbb{R}^n)$ and $|\phi| \le 1$, it follows that
\begin{align*}
    \phi \,  \div \tilde{\FF}_{\epsilon_j} \stackrel{*}{\rightharpoonup} \phi \,  \div \tilde{\FF},
    \qquad |\phi \,  \div \tilde{\FF}|(\partial E_k)=0,
\end{align*}
so that
\begin{align}\label{634}
\lim_{j \rightarrow \infty} \int_{E_k} \phi \, \dr \div \tilde{\FF}_{\epsilon_j}
=\int_{E_k} \phi \, \dr \div \tilde{\FF}.
\end{align}
Using \eqref{621}, \eqref{634}, and the fact that $\mathscr{H}^{n-1}(\partial E_k \setminus \partial^* E_k)=0$ and  $\div \tilde{\FF}=\div \FF$ on $\Omega$,
and letting $j \rightarrow \infty$ in \eqref{633} yield
\begin{align}\label{jingqi2}
\int_{E_k} \FF \cdot \nabla \phi\, \dr y
= -\int_{E_k} \phi \, \dr \div \FF-\int_{\partial^* E_k} \phi \FF \cdot \nu_{E_k}\dr \mathscr{H}^{n-1}.
\end{align}
Letting $k \rightarrow \infty$ in \eqref{jingqi2}, we have
\begin{align*}
\Big|\int_{\Omega} \FF \cdot \nabla \phi\, \dr y\Big|
=&\,\Big|\lim_{k \rightarrow \infty} \int_{E_k} \FF \cdot \nabla \phi \,\dr y\Big|\\
\le &\, \limsup_{k \rightarrow \infty} \Big(\Big |\int_{E_k} \phi \, \dr \div \FF \Big|+\Big|\int_{\partial^* E_k} \phi \FF \cdot \nu_{E_k} \dr\mathscr{H}^{n-1}\Big| \Big)\\
\le &\, |\div \FF|(\Omega)+ \norm{\FF}_{L^\infty(\Omega)} \sup_k P(E_k)\\
\lesssim_n &\, |\div \FF|(\Omega)+\norm{\FF}_{L^{\infty}(\Omega)}\mathscr{H}^{n-1}(\partial \Omega \setminus \Omega^0).
\end{align*}
Then
\begin{align*}
|\div \tilde{\FF}|(\mathbb{R}^n)
 =& \,\sup \Big\{-\int_{\mathbb{R}^n}  \FF \cdot \nabla\phi\,\dr y\,:\, \phi \in C_c^1(\mathbb{R}^n), |\phi|\le1 \Big\} \\
\lesssim_n &\, |\div \FF|(\Omega)+\norm{\FF}_{L^\infty(\Omega)}\, \mathscr{H}^{n-1}(\partial \Omega \setminus \Omega^0)<\infty.
\end{align*}
This completes the proof.
\end{proof}

\subsection{Weak convergence of the trace operator}
Let $\FF_j, \FF \in \mathcal{DM}^{p}(\Omega)$ for  $1 \le p \le \infty$
and $j=1,2,\cdots$,
and let $\Omega$ be a bounded open set.
From the definition of normal traces (see Definition \ref{trac}), we have
\begin{equation}\label{entodo}
|\ban{\FF \cdot \nu, \phi}_{\partial \Omega}|
\le \norm{\phi}_{L^{\infty}(\Omega)}|\div \FF|(\Omega)+\norm{\FF}_{L^p(\Omega)}\norm{\nabla \phi}_{L^{p'}(\Omega)}
\qquad \text{ for any $\phi \in \Lip_{c}(\R^{n})$},
\end{equation}
where $p'$ is the conjugate to $p$, {\it i.e.} $\frac{1}{p'}+\frac{1}{p}=1$.

Since \eqref{entodo} holds especially for any $\phi \in \Lip_c(\Omega)$,
the normal trace $\ban{\FF \cdot \nu,\, \cdot}$ is also a distribution in $\Omega$.
From Definition \ref{trac}, it follows that
$$
\ban{\FF \cdot \nu, \phi}_{\partial \Omega}= \div (\phi \FF)(\Omega)
$$
so that
the trace can be extended
to a functional in the dual of the space:
\begin{align*}
    X:= \{\phi \in C^0(\Omega) \cap L^{\infty}(\Omega)\,:\, \nabla \phi \in L^{p'}(\Omega)\},
\end{align*}
with norm $\|\phi\|_X=\|\phi\|_{L^\infty(\Omega)}+\|\nabla\phi\|_{L^{p'}(\Omega)}$.
We now introduce the following definition:

\begin{definition}\label{coushu2}
We say that $\ban{\FF_j \cdot \nu,\, \cdot}_{\partial \Omega}$ converges in the weak* topology to
$\ban{\FF \cdot \nu,\, \cdot}_{\partial \Omega}$,
i.e. $\ban{\FF_j \cdot \nu,\, \cdot}_{\partial \Omega} \stackrel{*}{\rightharpoonup} \ban{\FF \cdot \nu,\, \cdot}_{\partial \Omega}$,
provided that, for any $\phi \in X$,
\begin{equation}
\ban{\FF_j \cdot \nu, \phi}_{\partial \Omega} \to \ban{\FF \cdot \nu, \phi}_{\partial \Omega}\qquad\,\,\text{as $j\to\infty$}.
\end{equation}
\end{definition}

In this subsection, we prove that, if $\FF_j$ and $\FF$ are as in \eqref{coushu-a}--\eqref{coushu-b},
then the normal trace sequence of $\FF_j$ converges in the weak-star topology
to the normal trace of $\FF$.
This result will be used to prove the Gauss-Green formula up to the boundary; see Theorem \ref{da} in \S 5 below.

\begin{lemma}\label{i1}
Let $\FF_j, \FF\in \mathcal{DM}^{p}(\Omega)$, where $\Omega$ is a bounded open set, $1\le p \le \infty$,
and $j=1,2, \cdots$.
If
\begin{equation}\label{4.14a}
 |\div \FF_j|(\Omega) \rightarrow |\div \FF|(\Omega),
\quad \norm{\FF_j-\FF}_{L^1(\Omega)} \rightarrow 0 \qquad\, \mbox{as $j\to \infty$}
\end{equation}
with
$\sup_{j}\|\FF_j\|_{L^p(\Omega)}<\infty$,
then there exists a subsequence $\FF_{j_k}$ such that, for any $\phi \in X$,
\begin{align}
    \int_{\Omega} \phi \, \dr \div \FF_{j_k}
    \rightarrow  \int_{\Omega} \phi \, \dr \div \FF \qquad\,\, \mbox{as $k \rightarrow \infty$.}
\end{align}
\end{lemma}

\begin{proof}
Let $d(x):=\text{dist}(x, \partial \Omega)$ be the standard distance function.
Define
$$
U^{\ve}:=\{x \in \Omega\,: \, d(x)> \ve \}.
$$
For {\it a.e.} $\ve>0$, $U^{\ve}$ is a set of finite perimeter.
Clearly, $U^{\ve} \Subset \Omega$, $U^{\ve} \rightarrow \Omega$ in $L^1$,
and there exists a subsequence (still denoted as) $\{\FF_j\}$  such that, for {\it a.e.} $\ve>0$,
\begin{align}\label{d0}
   \lim_{j \rightarrow \infty} \int_{\partial U^{\ve}} |\FF_j-\FF|\, \dr\mathscr{H}^{n-1}=0.
\end{align}

Indeed, \eqref{d0} can be derived from the coarea formula (see \cite [Theorem 2.1]{CCT} and
the fact that $|\nabla d (x)|=1$ for {\it a.e.} $x \in \Omega$ (see \cite[Lemma 5.1]{CCT}).
Since  $\div \FF$ is a finite measure, we also have
\begin{equation}
\label{d1}
    |\div \FF_j |(\partial U^{\ve})= |\div \FF|(\partial U^{\ve})=0
    \qquad  \text{for {\it a.e.} } \ve>0, \,\, j=1,2,\cdots.
\end{equation}

After possibly discarding a set of $\mathcal{L}^1$-measure zero, the following Gauss-Green formulas hold
for {\it a.e.}\, $\ve>0$ (see \cite[Lemma 5.3 and Remark 5.10]{CCT} for a detailed proof):
\begin{align}
\label{d2}
    \int_{U^{\ve}} \phi \, \dr \div \FF= -\int_{U^{\ve}} \FF \cdot \nabla \phi\,\dr x
     - \int_{\partial U^{\ve}} \phi \FF \cdot \nu_{E_k} \dr\mathscr{H}^{n-1},
\end{align}
and
\begin{align}
\label{d3}
    \int_{U^{\ve}} \phi \, \dr \div \FF_j
    = -\int_{U^{\ve}} \FF_j \cdot \nabla \phi\,\dr x
    - \int_{\partial U^{\ve}} \phi \FF_j \cdot \nu_{E_k} \dr\mathscr{H}^{n-1}.
\end{align}

Let $\ve_k \to 0$ be a sequence so that \eqref{d0}--\eqref{d3} hold when $\ve=\ve_k, k=1,2, \cdots$.
Since $\div \FF$ is a measure, then
\begin{align}
    \label{d4}
    |\div \FF|(\Omega \setminus U^{\ve_k})< \frac{1}{k} \qquad\,\,\, \mbox{for $k$ large enough}.
\end{align}
Since $\lim_{j\rightarrow \infty} |\div \FF_j|(\Omega)=|\div \FF|(\Omega)$, it follows from \eqref{d1} that,
for each $k$,
\begin{align}
\label{d5}
    \lim_{j\rightarrow \infty} |\div \FF_j|(\Omega \setminus U^{\ve_k})=|\div \FF|(\Omega\setminus U^{\ve_k}).
\end{align}
Therefore, by \eqref{d0} and \eqref{d5}, we can choose $j_k$ such that
\begin{align}
&|\div \FF_{j_k}|(\Omega \setminus U^{\ve_k})< |\div \FF|(\Omega \setminus U^{\ve_k})+\frac{1}{k} < \frac{2}{k}, \label{d6}\\
&\int_{\partial U^{\ve_k}}|\FF_{j_k}-\FF|\,\dr \mathscr{H}^{n-1} < \frac{1}{k}, \label{d7}\\
&\FF_{j_k} \rightharpoonup \FF \,\,\,\mbox{in $L^p$ for $p\in (1,\infty)$}
\,\,\,\,\mbox{and}\,\,\,\, \FF_{j_k} \stackrel{*}{\rightharpoonup} \FF \,\,\,\mbox{in $L^\infty$} \qquad\,\, \mbox{as $k\to \infty$}.\label{4.22a}
\end{align}

Finally, from \eqref{d2}--\eqref{d3} and \eqref{d6}--\eqref{d7}, we have
\begin{align*}
    &\Big|\int_{\Omega} \phi \,\dr  \div \FF_{j_k} - \int_{\Omega} \phi \, \dr \div \FF\Big| \\
    &\le  \norm{\phi}_{L^{\infty}(\Omega)} \big( |\div \FF_{j_k}|(\Omega \setminus U^{\ve_k})+|\div \FF|(\Omega \setminus U^{\ve_k})\big)
     +\Big|\int_{U^{\ve_k}} \phi \,\dr \div \FF_{j_k} - \int_{U^{\ve_k}} \phi \, \dr \div \FF\Big|\\
    &\le  \frac{3}{k}\norm{\phi}_{L^{\infty}(\Omega)}+\Big|\int_{U^{\ve_k}} (\FF-\FF_{j_k}) \cdot \nabla \phi\, \dr x
     - \int_{\partial U^{\ve_k}} \phi (\FF_{j_k}-\FF)\cdot \nu_{E_k} \, \dr \mathscr{H}^{n-1}\Big|\\
    &\le  \frac{4}{k}\norm{\phi}_{L^{\infty}(\Omega)}+\Big|\int_{U^{\ve_k}} (\FF-\FF_{j_k}) \cdot \nabla \phi\, \dr x\Big|\\
    &\rightarrow  0 \qquad \mbox{as $k \rightarrow \infty$,}
\end{align*}
where we have used \eqref{4.14a} and \eqref{4.22a} in the last limit.
\end{proof}

From Lemma \ref{i1}, we have the following result:

\begin{theorem}\label{contin}
Let $\FF_j, \FF\in \mathcal{DM}^{p}(\Omega)$, where $\Omega$ is a bounded open set, $1\le p \le \infty$,
and $j=1,2,\cdots$.
If
\begin{equation*}
 |\div \FF_j|(\Omega) \rightarrow |\div \FF|(\Omega),
\quad \norm{\FF_j-\FF}_{L^1(\Omega)} \rightarrow 0 \qquad\, \mbox{as $j\to \infty$}
\end{equation*}
with
$\sup_{j}\|\FF_j\|_{L^p(\Omega)}<\infty$,
then
\begin{align*}
    \ban{\FF_j \cdot \nu,\, \cdot}_{\partial \Omega} \stackrel{*}{\rightharpoonup} \ban{\FF \cdot \nu,\, \cdot}_{\partial \Omega}
    \qquad \mbox{as $j\to \infty$}.
\end{align*}
In particular, this result holds for $\FF$ and $\FF_j$ given as in
Proposition {\rm \ref{approximation}}.
\end{theorem}

\begin{proof}
Given $\phi \in  X$, by Lemma \ref{i1}, we find that, for any subsequence of $\{\FF_j\}$,
there exists another subsequence $\{\FF_{j_k}\}$ such that
\begin{align*}
    \lim_{k \rightarrow \infty} \int_{\Omega} \phi \, \dr \div \FF_{j_k}=\int_{\Omega} \phi \, \dr \div \FF,
\end{align*}
so that
\begin{align*}
    \lim_{k \rightarrow \infty} \Big(\int_{\Omega} \phi \, \dr \div \FF_{j_k}
    + \int_{\Omega} \FF_{j_k} \cdot \nabla\phi\,\dr x \Big)=\int_{\Omega} \phi \, \dr \div \FF+ \int_{\Omega} \FF \cdot \nabla\phi\, \dr x.
\end{align*}
That is, as $k \rightarrow \infty$,
\begin{align*}
    \ban{\FF_{j_k} \cdot \nu, \phi}_{\partial \Omega} \to \ban{\FF \cdot \nu,\, \phi}_{\partial \Omega}.
\end{align*}
Since the limit is unique, we conclude
\begin{align*}
    \ban{\FF_{j} \cdot \nu,\, \cdot}_{\partial \Omega} \stackrel{*}{\rightharpoonup} \ban{\FF \cdot \nu,\, \cdot}_{\partial \Omega}.
\end{align*}
\end{proof}

\section{Gauss-Green Formula up to the Boundary on Extension Domains for Bounded Divergence-Measure Fields}
We begin this section with the following simple example, which says that,
if $\FF \in {\DM}^{\infty}(\Omega)$, and $\Omega$ is a bounded open set,
then $\ban{\FF \cdot \nu,\, \cdot}_{\partial \Omega}$ may not be concentrated on $\partial^* \Omega$.

\begin{example}\label{classical}
Let $\Omega=D \setminus S$, where $D=(-1,1) \times (-1,1)$ and $S=(-1,1) \times \{0\}$. We define
\begin{equation}
    \FF(x_1,x_2):= \begin{cases}
    (0,1)  \quad &\mbox{for $x_2>0$},\\
    (0,-1) \quad &\mbox{for $x_2<0$}.
    \end{cases}
\end{equation}
Let $\Omega^+=D \cap \{x_2>0\}$ and $\Omega^-=D \cap \{x_2<0\}$,
and let $S_1:=(-1,1) \times \{1\}$ and $S_2:=(-1,1) \times \{-1\}$.
Then, for any $\phi \in C_c^1(\mathbb{R}^2)$, we have
\begin{align*}
    \int_{\Omega} \FF \cdot \nabla \phi\, \dr x
    =& \int_{\Omega^+} \FF \cdot \nabla \phi\, \dr x +\int_{\Omega^-} \FF \cdot \nabla \phi\, \dr x\\
    =&\int_{\Omega^+} \partial_{x_2} \phi \, \dr x -\int_{\Omega^-} \partial_{x_2} \phi\, \dr x\\
    =& \int_S (-\phi)\, \dr\mathscr{H}^{n-1}
       -\int_S \phi\, \dr\mathscr{H}^{n-1} + \int_{S_1 \cup S_2} \phi\,\dr\mathscr{H}^{n-1}\\
    =& -2 \int_S \phi\, \dr\mathscr{H}^{n-1}+\int_{S_1 \cup S_2} \phi\, \dr\mathscr{H}^{n-1},
\end{align*}
where we have used the classical Gauss-Green formula.

Since $\div \FF=0$ on $\Omega$, the previous computation yields
\begin{align*}
    \ban{ \FF \cdot \nu,\,\phi}_{\partial \Omega}
    =\int_{\Omega} \phi \, \dr \div \FF
     +\int_{\Omega} \FF \cdot \nabla \phi\, \dr x
     = -2 \int_S \phi\, \dr\mathscr{H}^{n-1}
       +\int_{S_1 \cup S_2} \phi\, \dr\mathscr{H}^{n-1}.
\end{align*}
Therefore, $\ban{\FF \cdot \nu,\, \cdot}_{\partial \Omega}$ is
a measure $\mu:=-2 \mathscr{H}^1 \res S+\mathscr{H}^{1}\res({S_1 \cup S_2})$.
\end{example}

Motivated by Example \ref{classical}, in order to study the normal trace $\ban{\FF \cdot \nu,\, \cdot}_{\partial \Omega}$
for a bounded divergence-measure field $\FF$ and  an extension domain $\Omega$,
the measure-theoretic interior part of the topological boundary has to be considered.
This example has motivated us to study the characterization of domains satisfying \eqref{neat-2}
and to formulate and prove the following theorem.

\begin{theorem}\label{da}
Let $\Omega$ be a bounded open set of finite perimeter that is an extension domain for $\FF \in \mathcal{DM}^{\infty}(\Omega)$,
and let $\tilde{\FF}$ be as in \eqref{ext}.
Then the trace operator $\ban{\FF \cdot \nu,\, \cdot}_{\partial \Omega}$ is a finite Radon
measure $\mu$ concentrated on $\partial \Omega \setminus \Omega^0$ with
\begin{align}\label{l3}
   \mu=-\div \tilde{\FF} \res\left((\partial \Omega \cap \Omega^1) \cup \partial ^* \Omega\right)
     = -\div \tilde{\FF} \res(\partial \Omega \cap \Omega^1) -2 \overline{\tilde{\FF} \cdot D\chi_{\Omega}},
\end{align}
where $\overline{\tilde{\FF} \cdot D\chi_{\Omega}}$ is a measure concentrated
on $\partial^* \Omega$ defined as in Theorem {\rm \ref{productruleinfty}}.

Since  $\mu \ll \mathscr{H}^{n-1} \res{\left((\partial \Omega \cap \Omega^1) \cup \partial^* \Omega\right)}$,
there exists $g \in L^{1}\left(\partial \Omega \setminus \Omega^0 ; \mathscr{H}^{n-1} \right)$ such that
the following Gauss-Green formula {\it up to the boundary} holds{\rm :}
\begin{equation}
\label{nuevo}
\int_{\Omega} \phi\, \dr\div \FF + \int_{\Omega} \FF \cdot \nabla \phi\, \dr x
=\int_{\partial \Omega \setminus \Omega^0} \phi(x)g(x)\, \dr \mathscr{H}^{n-1}(x)
\end{equation}
for any $\phi \in C_c^1(\mathbb{R}^n)$.
In particular, if $\Omega$ is a bounded open set satisfying \eqref{neat-2},
then \eqref{nuevo} holds and, in this case, $g \in L^{\infty} \left( \partial  \Omega \setminus \Omega^0; \mathscr{H}^{n-1} \right)$.
\end{theorem}

\begin{proof}
For any $\phi \in C_c^{1}(\mathbb{R}^n)$, using the classical product rule:
\begin{align}
\label{prod''}
    \div (\phi \tilde{\FF})=\phi\,\div \tilde{\FF}+ \nabla \phi \cdot \tilde{\FF},
\end{align}
which is a particular case of \eqref{PRDMF}, we have

\begin{align}\label{l4}
    \int_{\mathbb{R}^n} (\chi_{\Omega} * \rho_{\epsilon}) \phi\, \dr \div \tilde{\FF}+
     \int_{\mathbb{R}^n} (\chi_{\Omega} * \rho_{\epsilon}) \nabla \phi \cdot \tilde{\FF}\, \dr x
     =-\int_{\mathbb{R}^n} \phi \tilde{\FF} \cdot \nabla (\chi_{\Omega} * \rho_{\epsilon})\, \dr x.
\end{align}
Since, for $\mathscr{H}^{n-1}-a.e.\,\, x \in \mathbb{R}^n$,
\begin{equation}
    \lim_{\epsilon \rightarrow 0} (\chi_{\Omega} * \rho_{\epsilon})(x)= \begin{cases}
1 \quad &\mbox{for $x \in \Omega^1$},\\
\frac{1}{2} \quad &\mbox{for $x \in \partial^* \Omega$},\\
0 \quad &\mbox{for $x \in \Omega^0$},
\end{cases}
\end{equation}
and $\div \tilde{\FF} \ll \mathscr{H}^{n-1}$ as stated in \eqref{cui1},
letting $\epsilon \rightarrow 0$ in \eqref{l4} yields
\begin{align}\label{l9}
    \int_{\Omega} \phi \, \dr \div \tilde{\FF}
    +\int_{\Omega} \tilde{\FF} \cdot \nabla \phi\, \dr x
    +\int_{\Omega^1 \cap \partial \Omega}\phi \, \dr \div \tilde{\FF}
    +\frac{1}{2}\int_{\partial^* \Omega} \phi \, \dr \div \tilde{\FF}
    =-\int_{\mathbb{R}^n} \phi\, \dr \overline{\tilde{\FF} \cdot D\chi_{\Omega}},
\end{align}
so that
\begin{align}\label{l6}
    -\int_{\Omega} \tilde{\FF} \cdot \nabla \phi\, \dr x
    = \int_{\Omega} \phi \, \dr \div \tilde{\FF}
      +\int_{\Omega^1 \cap \partial \Omega}\phi \, \dr \div \tilde{\FF}
      +\frac{1}{2}\int_{\partial^* \Omega} \phi \, \dr \div \tilde{\FF}
      +\int_{\mathbb{R}^n} \phi\, \dr \overline{\tilde{\FF} \cdot D\chi_{\Omega}}.
\end{align}
It is well known that $\overline{\tilde{\FF} \cdot D\chi_{\Omega}}$ is a finite Radon measure
concentrated on $\partial ^* \Omega$ (see \cite[Page 251]{ChenTorres} for a proof).
Therefore, it follows from \eqref{l6} that $\div \tilde{\FF}$ is a measure concentrated
on $\Omega^1 \cup \partial^* \Omega$.
Since
\begin{align}\label{l7}
    \Omega^1 \cup \partial^* \Omega=\Omega \cup (\Omega^1 \cap \partial \Omega) \cup \partial^*\Omega,
\end{align}
we have
\begin{align*}
    -\int_{\Omega} \tilde{\FF} \cdot \nabla \phi\, \dr x
    = & -\int_{\mathbb{R}^n}  \tilde{\FF} \cdot \nabla \phi\, \dr x\\
    =& \int_{\mathbb{R}^n} \phi \, \dr \div \tilde{\FF}\\
    =& \int_{\Omega} \phi \, \dr\div \tilde{\FF}
      + \int_{\Omega^1 \cap \partial \Omega}\phi \, \dr \div \tilde{\FF}
       +\int_{\partial^* \Omega} \phi \, \dr \div \tilde{\FF},
\end{align*}
which, together with \eqref{l6}, implies
\begin{align}\label{bagua}
    \div \tilde{\FF} \res{\partial ^* \Omega}=2 \overline{\tilde{\FF} \cdot D\chi_{\Omega}}.
\end{align}

Since $\div \tilde{\FF} \res{\Omega}=\div \FF \res{\Omega}$,
the definition of normal traces and \eqref{l9} imply that
$\ban{\FF \cdot \nu, \cdot}_{\partial \Omega}$ is a measure $\mu$  concentrated
on $(\partial \Omega \cap \Omega^1) \cup \partial ^* \Omega$ so that
\begin{align} \label{l10}
    \int_{\mathbb{R}^n} \phi\, \dr \mu
    =-\int_{\Omega^1 \cap \partial \Omega}\phi \, \dr \div \tilde{\FF}
    -\frac{1}{2}\int_{\partial^* \Omega} \phi \, \dr \div \tilde{\FF}
    -\int_{\partial^* \Omega} \phi\, \dr \overline{\tilde{\FF} \cdot D\chi_{\Omega}}.
\end{align}
A combination of
\eqref{l7}--\eqref{l10} gives \eqref{l3}.

Since $\Omega$ is an extension domain, $\div \tilde{\FF}$ is a Radon measure.
By \eqref{cui1}, $\mu$ is a finite measure and $\mu \ll \mathscr{H}^{n-1} \res{\left((\partial \Omega \cap \Omega^1) \cup \partial^* \Omega\right)}$
so that, by the Riesz representation theorem and Theorem \ref{Federertheorem},
there exists $g \in L^{1}\left( \partial \Omega \setminus \Omega^0; \mathscr{H}^{n-1} \right)$ such that
\begin{align}
    \int_{\mathbb{R}^n} \phi\, \dr \mu
    =\int_{\partial \Omega \setminus \Omega^0} g\phi\, \dr\mathscr{H}^{n-1}.
\end{align}

If $\Omega$ satisfies \eqref{neat-2}, then, by Theorem \ref{ahax},
$\Omega$ is an extension domain for bounded divergence-measure fields
so that the above results hold.
In order to show that $g \in L^{\infty}(\partial \Omega \setminus \Omega^0; \mathscr{H}^{n-1})$ in this case,
we use the Lebesgue differentiation theorem and the argument in the proof of Theorem \ref{brandnew} as follows:

Given $x \in \partial \Omega \setminus \Omega^0$,
then, for any $r>0$, proceeding as in the proof of Theorem \ref{brandnew},
especially the derivation of \eqref{tiao'},
there exist  $E_k \Subset \Omega$ such that $E_k \rightarrow \Omega$ in $L^1$ and
\begin{align} \label{tiao}
P(E_k; B_r(x))  \lesssim_n \mathscr{H}^{n-1} \left((\partial \Omega \setminus \Omega^0) \cap B_r(x) \right).
\end{align}
Let $\FF_j$ be the sequence given in Proposition \ref{approximation}.
Since $\FF \in \mathcal{DM}^{\infty}(\Omega)$, the Gauss-Green formulas \eqref{d2}--\eqref{d3}
hold for any $\phi \in C_c^1(\mathbb{R}^n)$ so that Theorem \ref{contin} gives
\begin{equation}\label{porfavor}
 \ban{\FF_j \cdot \nu, \phi}_{\partial \Omega} \to \ban{\FF \cdot \nu, \phi}_{\partial \Omega}=\int_{\R^n} \phi \, \dr\mu.
\end{equation}
Then, for any $\phi \in C_c^1(B_r(x))$ with $|\phi| \le 1$, we compute
\begin{align*}
    \int_{\R^n} \phi \, \dr\mu
    &= \lim_{j \rightarrow \infty} \ban{\FF_j \cdot \nu, \phi}_{\partial \Omega}\\
    & = \lim_{j \rightarrow \infty} \int_{\Omega} \dr \div ( \phi \FF_j)\\
    & = \lim_{j \rightarrow \infty} \lim_{k \rightarrow \infty} \int_{E_k} \dr \div ( \phi \FF_j)\\
    &= \lim_{j \rightarrow \infty} \lim_{k \rightarrow \infty}\Big( -\int_{\partial^* E_k \cap B_r(x)} \phi \FF_j \cdot \nu_{E_k} \,\dr \mathscr{H}^{n-1}\Big)\\
    & \le \lim_{j \rightarrow \infty} \lim_{k \rightarrow \infty} \norm{\FF_j}_{\infty} P(E_k; B_r(x))\\
    &\lesssim_n  \norm{\FF}_{\infty} \mathscr{H}^{n-1} \left((\partial \Omega \setminus \Omega^0) \cap B_r(x) \right),
\end{align*}
where we have used  \eqref{633}, \eqref{tiao}--\eqref{porfavor}, and the fact that $E_k \rightarrow \Omega$ in $L^1$.
Therefore, for $\mathscr{H}^{n-1}$--{\it a.e.} $x \in \partial \Omega \setminus \Omega^0$,
the Lebesgue differentiation theorem yields
\begin{align*}
    |g(x)|=\lim_{r \rightarrow 0} \frac{|\mu(B_r(x))|}{\mathscr{H}^{n-1} \left(\left(\partial \Omega \setminus \Omega^0 \right) \cap B_r(x) \right)} \lesssim_n \norm{\FF}_{\infty}.
       \end{align*}
 This completes the proof.
\end{proof}

Using the same method as in Theorem \ref{da},
we now provide a new elementary proof (see Theorem \ref{GGF} below) of the
Gauss-Green formula \eqref{mia}

We recall the following product rule, which is a particular case of \eqref{PRDMF}:
\begin{proposition}\label{prod'}
Let $\GG \in \mathcal{DM}^{\infty}(\Omega)$ and $\phi \in C_c^1(\Omega)$,
then $\div (\phi \GG) \in \mathcal{DM}^{\infty}(\Omega)$ with
\begin{align}\label{prod}
    \div (\phi \GG)=\phi\,\div\,\GG+ \nabla \phi \cdot \GG.
\end{align}
\end{proposition}

Then we have the following key proposition:
\begin{proposition}\label{hkey}
Let $\Omega \subset \mathbb{R}^n$ be an open set, and
let $E$ be a set of finite perimeter with $E \Subset\Omega$.
Assume that $\GG \in \mathcal{DM}^{\infty}(\Omega)$.
Then
\begin{align}\label{jh}
    \int_{E^1} \dr\div (\phi \GG)
    +\frac{1}{2} \int_{\partial^* E} \dr\div (\phi \GG)
    =-\int_{\partial^* E} \phi\, \dr \, \overline{\GG \cdot D\chi_{E}}
    \quad\,\mbox{for any $\phi \in C_c^1(\Omega)$}.
\end{align}
Equivalently,
\begin{align}    \label{h4'}
    \int_{E} \GG \cdot \nabla \phi\, \dr x
     +\int_{E^1} \phi  \, \dr \div\,\GG
     +\frac{1}{2}\int_{\partial^* E}\phi  \, \dr \div\,\GG
     =-\int_{\partial^* E} \phi\, \dr\, \overline{\GG \cdot D\chi_{E}}
\end{align}
for any $\phi \in C_c^1(\Omega)$.
\end{proposition}

\begin{proof}  Notice that
\begin{align}\label{l4'}
    \int \chi_{E} * \rho_{\epsilon} \, \dr\div (\phi \GG)
    =-\int \phi\, \GG \cdot \nabla (\chi_{E} * \rho_{\epsilon})\, \dr x.
\end{align}
Since,  for $\mathscr{H}^{n-1}-a.e. \,\, x \in \mathbb{R}^n$,
\begin{equation}
\lim_{\epsilon \rightarrow 0} (\chi_{E} * \rho_{\epsilon})(x)
= \begin{cases}
1 \quad & \mbox{for $x \in E^1$},\\
\frac{1}{2} \quad &\mbox{for  $x \in \partial^* E$},\\
0 \quad & \mbox{for $x \in E^0$},
\end{cases}
\end{equation}
and \eqref{cui1} holds,
then letting $\epsilon \rightarrow 0$ in \eqref{l4'} yields \eqref{jh},
where we have used the known fact that $\overline{\GG \cdot D\chi_{E}}$
is a bounded measure concentrated on $\partial^* E$; see \cite[page 251]{ChenTorres}
for a proof.
Note that \eqref{h4'} is equivalent to \eqref{jh} in view of \eqref{prod}
and the fact that $|E^1 \Delta E|=0$.
\end{proof}

\begin{corollary} \label{chais}
Let $E \Subset\Omega$ and $\GG \in \mathcal{DM}^{\infty}(\Omega)$ with $\GG=0$ outside $E$.
Then $\div\, \GG$ is a measure concentrated on $E^1 \cup \partial^* E${\rm :}
\begin{align}\label{spt}
     \div\,\GG =\div\, \GG \res( E^1 \cup \partial^* E).
\end{align}
\end{corollary}

\begin{proof}
For any $\phi \in C_c^1(\mathbb{R}^n)$, we employ \eqref{h4'} to obtain
 \begin{align*}
     \int_{\mathbb{R}^n} \phi  \, \dr \div\,\GG
     =&\,-\int_{\mathbb{R}^n} \GG \cdot \nabla \phi\, \dr x
     = -\int_E \GG \cdot \nabla \phi \, \dr x\\
     =&\,\int_{E^1} \phi  \, \dr \div\, \GG
      +\frac{1}{2}\int_{\partial^* E}\phi  \, \dr \div\,\GG
      + \int_{\partial^* E} \phi\, \dr \overline{\GG \cdot D\chi_{E}}.
 \end{align*}
 This completes the proof.
\end{proof}

We note that, in the case of $BV$ functions, we have the following similar result, as shown in
\cite[Theorem 3.84]{afp}.
\begin{remark}
Let $f \in BV(\Omega)$ and $E \Subset \Omega$. If $f=0$ on $E^c$, then
\begin{align}\label{cy7}
    Df = Df \res(E^1 \cup \partial^* E).
\end{align}
\end{remark}

\begin{remark}
The result above does not require the boundedness of BV functions,
thanks to the coarea formula.
It would be interesting to prove Corollary {\rm \ref{chais}} for divergence-measure fields
without the boundedness assumption.
\end{remark}

The following statement is a consequence of Proposition \ref{hkey} and Corollary \ref{chais}.
\begin{lemma}\label{bangz}
Let $E \Subset\Omega$ and $\GG \in \mathcal{DM}^{\infty}(\Omega)$.
Then
\begin{align}\label{tuif}
  \int_{\partial^* E} \phi\, \dr \overline{(\chi_E-\chi_{E^c})\GG \cdot D\chi_E}
  =\frac{1}{2}\int_{\partial^* E} \dr\div (\phi \GG) \qquad\mbox{for any $\phi \in C_c^1(\Omega)$}.
\end{align}
\end{lemma}

\begin{proof}
By Corollary \ref{chais} and Proposition \ref{productruleinfty}, $\div (\phi \chi_E \GG)$ is a measure concentrated on $ E^1 \cup \partial ^* E$
so that
 \begin{align}\label{jin}
0=\int_{\Omega} \dr\div (\phi \chi_E \GG) = \int_{E^1 \cup \partial^* E} \dr\div (\phi \chi_E \GG).
 \end{align}
Hence, \eqref{jh} and \eqref{jin} imply
\begin{align}\label{dt1}
 \int_{\partial^* E} \phi\, \dr \overline{\chi_E\GG \cdot D\chi_E}
 =\frac{1}{2}\int_{\partial^* E} \dr\div (\phi \chi_E\GG).
\end{align}

On the other hand, $\chi_{E^c}\GG=0$ on $E$. Then, again by Corollary \ref{chais},
$ \div (\phi \chi_{E^c}\GG)$ is a measure on $(E^c)^1 \cup \partial^*(E^c) = E^0 \cup \partial^* E$.
With $\GG$ replaced by $\chi_{E^c}\GG$ in \eqref{jh}, the first term vanishes so that
\begin{align}\label{dt2}
     -\int_{\partial^* E} \phi\, \dr \overline{\chi_{E^c}\GG \cdot D\chi_E}
     =\frac{1}{2}\int_{\partial^* E} \dr\div (\phi \chi_{E^c}\GG).
\end{align}
Adding \eqref{dt1}--\eqref{dt2} together gives \eqref{tuif}.
\end{proof}

As a byproduct, the following result is immediate from Proposition \ref{prod'} and Lemma \ref{bangz}:

\begin{corollary}\label{aoye}
Let $\GG \in \mathcal{DM}^{\infty}(\Omega)$, and let $E \Subset\Omega$ be a set of finite perimeter.
Then the following identity holds{\rm :}
\begin{align}    \label{h3}
    \overline{(\chi_E-\chi_{E^c})\GG  \cdot D\chi_E}=\frac{1}{2}\div\, \GG\res{\partial^* E} .
\end{align}
\end{corollary}

Plugging \eqref{h3}, with $\GG$ replaced by $\FF$, into \eqref{h4'} immediately yields
the following Gauss-Green formula on sets of finite perimeter:

\begin{theorem}\label{GGF}
Let $\Omega \subset \mathbb{R}^n$ be an open set, and let $E$ be a set of finite perimeter
with $E \Subset \Omega$.
Assume that $\FF \in \mathcal{DM}^{\infty}(\Omega)$.
Then, for any $\phi \in C_c^1(\Omega)$,
\begin{align}\label{h4}
    \int_E \FF \cdot \nabla \phi\, \dr x
    +\int_{E^1} \phi \, \dr \div \FF
    =-2 \int_{\partial^* E} \phi\, \dr \overline{\chi_E \FF \cdot D\chi_{E}}.
\end{align}
\end{theorem}

\begin{remark}
It is well known that measure $2\overline{\chi_E \FF \cdot D\chi_E}$,
which is concentrated on $\partial^* E$,
corresponds to an $L^{\infty}$ function on $\partial^* E$.
This $L^{\infty}$ function is called the interior normal trace of $\FF$ on $\partial E$, denoted by $\mathfrak{F}_{\rm i} \cdot \nu_E$.
\end{remark}

Applying the previous arguments to $\tilde{\FF}$ and viewing $\Omega$ as a set that is compactly contained in $\R^n$,
we now give the following {\it up to the boundary} Gauss-Green formula corresponding to $\ban{\FF \cdot \nu,\, \cdot}_{\partial E^1}$
that does not require $E \Subset \Omega$:

\begin{theorem}
Let $\Omega$ be a bounded open set satisfying \eqref{neat-2}, and let $\FF \in \DM^{\infty}(\Omega)$.
Then, for any set $E \subset \Omega$ of locally finite perimeter and $\phi \in C_c^1(\mathbb{R}^n)$,
\begin{align*}
&\int_{E^{1}} \phi \, \dr \div \FF + \int_{E} \FF \cdot \nabla \phi \, \dr x
 =  -\int_{\redb E} \phi \, \mathfrak{F}_{\rm i} \cdot \nu_{E} \, \dr \Haus{n - 1},
\end{align*}
where $\mathfrak{F}_{\rm i} \cdot \nu_{E} \in L^{\infty}(\partial^*E; \mathscr{H}^{n-1})$
is the interior normal trace, which also satisfies
\begin{align}\label{cui3}
\int_{\redb E} \phi \, \mathfrak{F}_{\rm i} \cdot \nu_{E} \, \dr \Haus{n - 1}
=2 \int_{\partial^* E} \phi\, \dr \overline{\chi_E \FF \cdot D\chi_{E}}.
\end{align}
\end{theorem}
\begin{proof}
By Theorem \ref{GGF},
\begin{align*}
    \int_E \FF \cdot \nabla \phi \,\dr x +\int_{E^1} \phi \, \dr \div \FF
    =-2 \int_{\partial^* E} \phi\, \dr \overline{\chi_E \tilde{\FF} \cdot D\chi_{E}},
\end{align*}
where $\tilde{\FF}$ is as in \eqref{j1}. Since $E \subset \Omega$,
$\chi_E \tilde{\FF}=\chi_E \FF$. This implies \eqref{cui3}.
\end{proof}

\section{Solvability of the Divergence Equation \\ with Prescribed $L^\infty$ Normal Trace}

In \S\,5, we have shown that, if a bounded open set $\Omega$ satisfies \eqref{neat-2},
then, for any $\FF \in \mathcal{DM}^{\infty}(\Omega)$,
the normal trace of $\FF$ is an $L^{\infty}$ function $g$ concentrated on $\partial \Omega \setminus \Omega^0$ (see Theorem \ref{da}).
In the opposite direction, given $g \in L^{\infty}\left(\partial \Omega \setminus \Omega^0; \mathscr{H}^{n-1}\right)$,
we would like to know whether there exists $\FF \in \mathcal{DM}^{\infty}(\Omega)$ such that the normal trace of $\FF$ is $g$.
Thus, in this section, we consider the problem of solving the divergence equation with prescribed $L^{\infty}$ normal trace.
Let us first introduce the following definition.

\begin{definition}
Let $\Gamma \subset \mathbb{R}^n$.
We say that $\Gamma$ satisfies the upper $(n-1)$--Ahlfors regular condition
if there exists a constant $C>0$ such that, for any $x \in \Gamma$ and $r>0$,
\begin{align}\label{alf}
    \mathscr{H}^{n-1}\left(\Gamma \cap B_r(x)\right) \le Cr^{n-1}.
\end{align}
\end{definition}

Then we have
\begin{theorem} \label{util}
Let $\Omega$ be a bounded open set such that $\partial \Omega \setminus \Omega^0$
is $(n-1)$--Ahlfors regular.
Then, for any $g \in L^{\infty}\left(\partial \Omega \setminus \Omega^0; \mathscr{H}^{n-1}\right)$ with the compatibility condition{\rm :}
\begin{align}
    \int_{\partial \Omega \setminus \Omega^0} g \, \dr\mathscr{H}^{n-1}=0,
\end{align}
the problem of finding $\FF \in \mathcal{DM}^{\infty}(\Omega)$ such that
\begin{equation}\label{fanw}
    \begin{cases}
    \div \FF =0  \qquad \text{in } \Omega,\\[1mm]
     \ban{\FF \cdot \nu,\, \cdot}_{\partial \Omega} =g\in L^{\infty} (\partial \Omega \setminus \Omega^{0})
    \end{cases}
\end{equation}
is equivalent to the problem of finding $\FF \in \mathcal{DM}^{\infty}(\Omega)$ such that
\begin{equation}\label{fanw2}
    \begin{cases}
    \div \FF =0  \qquad\,\, \text{in } \Omega,\\[1mm]
     \ban{\FF \cdot \nu,\, \cdot}_{\partial \Omega}=h\in L^{\infty}(\partial^* \Omega) \quad \mbox{with $\int_{\partial^* \Omega} h \, \dr\mathscr{H}^{n-1}=0$}.
    \end{cases}
\end{equation}
\end{theorem}

\begin{proof}
Clearly, if problem \eqref{fanw} is solvable, then problem \eqref{fanw2} is also solvable.
We assume now that \eqref{fanw2} is solvable.

Let $\mu:= g\, \mathscr{H}^{n-1}\res(\partial \Omega \setminus \Omega^0)$.
Since $g \in L^{\infty}$, and \eqref{alf} holds, we see that, for any $x \in \mathbb{R}^n$,
\begin{align}\label{iii}
    |\mu|(B_r(x)) \le 2^{n-1} C \norm{g}_{\infty} \, r^{n-1} \qquad \mbox{for any $r>0$}.
\end{align}
Thus, for any $\phi \in C_c^1(\mathbb{R}^n)$ with $\phi \ge 0$, we have
\begin{align*}
    \Big|\int_{\mathbb{R}^n} \phi\, \dr \mu \bigg |
    \le &\, \int_0^{\infty}|\mu|\left(\{\phi > t\}\right)\dr t\\
    \lesssim_n &\, C \norm{g}_{\infty} \int_0^{\infty} P\left(\{\phi >t\}\right) \dr t\\
    \lesssim_n &\, C \norm{g}_{\infty} \int_{\mathbb{R}^n}|\nabla \phi| \, \dr x,
    \end{align*}
where we have used \eqref{iii} and the boxing inequality for the second inequality,
and the coarea formula for the third inequality.

For any $\phi \in C_c^1(\mathbb{R}^n)$,
we may write $\phi=\phi^+-\phi^-$ to conclude as above that
\begin{align*}
     \bigg|\int_{\mathbb{R}^n} \phi\, \dr\mu \bigg | \lesssim_n &\, C \norm{g}_{\infty} \int_{\mathbb{R}^n}|\nabla \phi|\,\dr x.
\end{align*}
Thus, by Phuc-Torres \cite[Theorem 3.3]{PT}, there exists $\GG \in L^{\infty}(\mathbb{R}^n; \mathbb{R}^n)$
such that
$$
\div\, \GG =-\mu= -g \mathscr{H}^{n-1}\res(\partial \Omega \setminus \Omega^0),
$$
that is, for any $\phi\in C_c^1(\mathbb{R}^n)$,
\begin{align}\label{6.7a}
    \int_{\mathbb{R}^n} \GG \cdot \nabla \phi\,\dr x
    =\int_{\partial \Omega \setminus \Omega^0} g\, \phi\, \dr\mathscr{H}^{n-1}.
\end{align}
From \eqref{6.7a}, we have
\begin{align}\label{6.8a}
\int_{\Omega} \GG \cdot \nabla \phi\, \dr x + \int_{\Omega^c} \GG \cdot \nabla \phi\, \dr x
=\int_{\partial \Omega \setminus \Omega^0} g \phi\, \dr \mathscr{H}^{n-1}
\end{align}
so that
\begin{align*}
    \ban{\GG \cdot \nu, \phi}_{\partial \Omega}
    =\int_{\Omega} \GG \cdot \nabla \phi\, \dr x
     +\int_{\Omega} \phi  \, \dr \div\,\GG =\int_{\Omega} \GG \cdot \nabla \phi\, \dr x,
\end{align*}
since $|\div\,G|(\Omega)=0$.
Thus, from \eqref{6.8a}, we conclude
\begin{align}\label{yuanx}
    \ban{\GG \cdot \nu, \phi}_{\partial \Omega}
    =\int_{\partial \Omega \setminus \Omega^0} g \phi\, \dr \mathscr{H}^{n-1}
     -\int_{\Omega^c} \GG \cdot \nabla \phi\, \dr x.
\end{align}

From Theorem \ref{GGF},
it follows that $\GG$ has an exterior normal trace $h \in L^{\infty}(\partial^* \Omega; \mathscr{H}^{n-1})$
such that, for any $\phi\in C_c^1(\mathbb{R}^n)$,
\begin{align}\label{6.11}
    \int_{\Omega^0} \GG \cdot \nabla \phi\, \dr x
    +\int_{\Omega^0} \phi  \, \dr \div\,\GG
    =\int_{\partial^* \Omega} h\, \phi \, \dr\mathscr{H}^{n-1}(y).
\end{align}
Since $|\div\,\GG|(\Omega^0)=0$ and
$\mathscr{L}^n(\Omega^c\setminus \Omega^0)=\mathscr{L}^n(\partial\Omega\setminus\Omega^0)=0$,
we conclude that \eqref{6.11} reduces to
\begin{align}\label{cy2}
    \int_{\Omega^c} \GG \cdot \nabla \phi\, \dr x
    =\int_{\partial^* \Omega} h\, \phi  \, \dr\mathscr{H}^{n-1}.
\end{align}

We define $\tilde{\GG}:=\GG \chi_B$, where $B$ is a large ball such that $\Omega\Subset B$. From Theorem 2.3, we have
$$
\div\,\tilde{\GG}=\chi_B^* \, \div\,\GG +\overline{\GG\cdot D\chi_B},
$$
where $\overline{\GG\cdot D\chi_B}$ is concentrated on $\partial B$ and $\chi^*_B\equiv 1$ on $B$.
Formulas \eqref{6.11}--\eqref{cy2} also hold for $\tilde{\GG}$.
Thus, for any $\phi\in C_c^1(\mathbb{R}^n)$,
\begin{align}\label{6.12a}
    \int_{\Omega^c} \tilde{\GG} \cdot \nabla \phi\, \dr x
    =\int_{\partial^* \Omega} h\,\phi  \, \dr\mathscr{H}^{n-1}.
\end{align}
Since $\tilde{\GG}\equiv 0$ outside $B$, we can choose a test function $\phi\in C_c^1(\mathbb{R}^n)$
with $\phi\equiv 1$ on $B$ for \eqref{6.12a} to obtain
$$
\int_{\partial^* \Omega} h \, \dr\mathscr{H}^{n-1}=0.
$$

This compatibility condition and our assumption that problem (6.4) is solvable imply
the existence of a vector field $\hat{\FF}$ such that $\div\,\hat{\FF}=0$ in $\Omega$ and
\begin{align}\label{6.13a}
    \langle \hat{\FF} \cdot \nu, \phi\rangle_{\partial \Omega}
    =\int_{\partial^*\Omega} h\, \phi\, \dr \mathscr{H}^{n-1}
    \qquad \mbox{for every $\phi\in C_c^1(\mathbb{R}^n)$}.
\end{align}

We now define
$\FF:=\GG + \hat{\FF}$,
and note that $\FF$ is a solution of \eqref{fanw}.
Indeed, it is clear that
$\div \FF=\div\,\GG + \div \hat{\FF}=0$ in $\Omega$.
Moreover, for any $\phi\in C_c^1(\mathbb{R}^n)$,
\begin{align*}
\langle \FF \cdot \nu, \phi\rangle_{\partial \Omega}
= & \int_{\Omega} \phi \, \dr \div\, \GG
 + \int_{\Omega} \GG\cdot \nabla \phi \, \dr x
 + \int_{\Omega} \phi \, \dr \div \hat{\FF}
  + \int_{\Omega} \hat{\FF} \cdot \nabla \phi \, \dr x\\
=\,& \langle\GG\cdot \nu, \phi\rangle_{\partial \Omega} + \langle \hat{\FF} \cdot \nu, \phi\rangle_{\partial \Omega}\\
=& \int_{\partial \Omega \setminus \Omega^0} g \phi\, \dr\mathscr{H}^{n-1}
 -\int_{\Omega^c} \GG \cdot \nabla \phi\, \dr x
   + \langle \hat{\FF} \cdot \nu, \phi\rangle_{\partial \Omega}\\
=&  \int_{\partial \Omega \setminus \Omega^0} g \phi\, \dr\mathscr{H}^{n-1}
  - \int_{\partial^* \Omega} \phi h \, \dr\mathscr{H}^{n-1}
    + \langle \hat{\FF} \cdot \nu, \phi\rangle_{\partial \Omega}\\
=& \int_{\partial \Omega \setminus \Omega^0} g \phi\, \dr\mathscr{H}^{n-1}
  - \int_{\partial^* \Omega} \phi h \, \dr\mathscr{H}^{n-1}
   + \int_{\partial^* \Omega} \phi h \, \dr\mathscr{H}^{n-1}\\
=& \int_{\partial \Omega \setminus \Omega^0} g \phi\, \dr\mathscr{H}^{n-1},
\end{align*}
where we have used \eqref{yuanx} for the third equality,
\eqref{cy2} for the fourth equality, \eqref{6.13a} for the fifth equality,
as well as the fact that  $\hat{\FF}$ solves \eqref{fanw2}.
This shows that the distributional normal trace
$\ban{\FF \cdot \nu,\, \cdot}_{\partial \Omega}=g\in L^\infty(\partial\Omega \setminus\Omega^0)$.
Therefore, $\FF$ solves (6.3).

\begin{remark}
Theorem {\rm \ref{util}} can be useful for the problems whose domains have interior fractures as it is the case
of the two-dimensional example $\Omega:= \{x: |x|<1,\,\,\mbox{$x_2 \neq 0$ when $x_1>0$}\}$
discussed in the introduction.
Given any data trace:
$$
g \in L^{\infty}(\{|x|=1\} \cup \{0<x_1<1,\, x_2=0\}),
$$
then the solution of \eqref{fanw} can be found, provided that we know how \eqref{fanw2} can be solved,
which has a simpler geometry since $\partial^* \Omega = \{|x|=1\}$ is just the unit circle.
\end{remark}
\end{proof}

\section{Applications and Remarks Related to Traces and Extension Domains for Bounded BV Functions and Divergence-Measure Fields}
In this section, we analyze extension domains for bounded $BV$ functions and show that \eqref{neat-2}
is a sufficient (but not a necessary)
condition for $\Omega$ to be an extension domain for bounded $BV$ functions. We also give some remarks on the traces and extension
domains for bounded BV functions and divergence-measure fields.

\subsection{Extension domains for bounded BV functions}
We can similarly define the extension domain for bounded BV functions.

\begin{definition}\label{zizao}
We say that an open set $\Omega$ is an extension domain for bounded BV functions if, for any $u \in BV(\Omega) \cap L^{\infty}(\Omega)$,
the corresponding function $\tilde{u}$,
defined as $u$ inside $\Omega$ and zero otherwise, also belongs to $BV(\mathbb{R}^n)$.
\end{definition}

Since divergence-measure fields are a generalization of $BV$ vector fields,
the following corollary is direct from Theorem \ref{ahax}:
\begin{corollary}\label{exbv}
Let $\Omega$ be an open set satisfying \eqref{neat-2}.
Then $\Omega$ is an extension domain for bounded BV functions. In particular,
\begin{align} \label{criteria}
\mbox{$P(E)<\infty$ $\qquad$ for any $E \subset \Omega$ with $P(E;\Omega)<\infty$. }
\end{align}
\end{corollary}

In fact, \eqref{neat-2} can directly imply \eqref{criteria}.
Indeed, let $E \subset \Omega$ be a set of finite perimeter in $\Omega$,
and let $\partial^m E$ denote the measure-theoretic boundary of $E$.
Since $E \subset \Omega$, we have $\partial^m E \cap \Omega^0= \emptyset$.
Thus, using
$$
\mathbb{R}^n=\Omega^1 \cup \partial^m \Omega \cup \Omega^0,\qquad
\Omega^1=\Omega \cup \left(\Omega^1 \cap \partial \Omega\right),
$$
we have
\begin{align*}
\mathscr{H}^{n-1}(\partial^mE)
=&\,\mathscr{H}^{n-1}(\partial^mE \cap \Omega^1)
  +\mathscr{H}^{n-1}(\partial^mE \cap \partial^m \Omega) \\
\le&\,\mathscr{H}^{n-1}(\partial^mE \cap \Omega)+\mathscr{H}^{n-1}(\partial^mE \cap \partial \Omega \cap \Omega^1)+P(\Omega)\\
 = &\, P(E;\Omega)+\mathscr{H}^{n-1}(\partial \Omega \setminus \Omega^0)<\infty.
\end{align*}
By Theorem \ref{criteria'}, $E$ is a set of finite perimeter in $\mathbb{R}^n$.

\medskip
Then a natural followup question is whether
 \eqref{neat-2} is equivalent to \eqref{criteria}.
In the rest of this subsection, we answer this question negatively,
by giving an example showing that there exists an open set $\Omega$ with \eqref{criteria},
but $\mathscr{H}^{n-1}(\partial \Omega \setminus \Omega^0)=\infty$.

We first introduce the so-called Sobolev extension domain.
\begin{definition}
We say that $\Omega$ is a Sobolev extension domain if, for any $u \in W^{1,p}(\Omega)$,
there is a bounded operator $E\,{\rm :} \,  W^{1,p}(\Omega) \rightarrow W^{1,p}(\mathbb{R}^n)$
and a constant $C(\Omega)>0$ such that
\begin{align*}
    Eu(x)=    u(x) \qquad \mbox{for all $x \in \Omega$},
\end{align*}
and
$$
\|Eu\|_{W^{1,p}(\mathbb{R}^n)} \le C(\Omega)\|u\|_{W^{1,p}(\Omega)}.
$$
\end{definition}

The Sobolev extension domains include
Lipschitz domains,
but can be much more general.
By \cite{GO,Jo},
a uniform domain is a Sobolev extension domain.
The uniform domains can have purely un-rectifiable boundary;
for example, the complement of $4$-corner Cantor set in a ball.
See also Definition \ref{uniformdomain} and Example \ref{aha6}
below for the definition and a concrete example of uniform domains.

The next proposition says that a Sobolev extension domain must be an extension domain for bounded BV functions.
\begin{proposition}\label{haof}
Let $\Omega$ be a Sobolev extension domain of finite perimeter.
Then $\Omega$ is an extension domain for bounded BV functions.
\end{proposition}

\begin{proof}
Let $u \in BV(\Omega)$.
By \cite[Theorem 3.9]{afp},
there exist $u_j \in C^{\infty}(\Omega), j=1,2, \cdots$,
such that $u_j \rightarrow u$ in $L^1(\Omega)$ and $|Du_j|(\Omega) \rightarrow |Du|(\Omega)$.
We may also assume that $\|u_j\|_{\infty} \le \|u\|_{\infty}$
and $|Du_j|(\Omega) \le 2|Du|(\Omega)$.
We extend $u_j \in W^{1,1}(\Omega)$ outside as $\bar{u}_j$, with
\begin{align*}
    \|\bar{u}_j\|_{W^{1,1}(\mathbb{R}^n)} \le C \|u_j\|_{W^{1,1}(\Omega)}
    =  C \|u_j\|_{BV(\Omega)}
    \le 2C\|u\|_{BV(\Omega)},
\end{align*}
where $C$ is the constant in the definition of Sobolev extension domains.

By the standard mollification,
we can actually choose $w_j \in C_c^{\infty}(\mathbb{R}^n)$
with $\|w_j\|_{\infty} \le \|u\|_{\infty}$
such that $w_j \rightarrow u$ in $L^1(\Omega)$ and $|Dw_j|(\mathbb{R}^n) \le C |Du|(\Omega)$.
Thus, for any $\phi \in C_c^1(\mathbb{R}^n, \mathbb{R}^n)$ with $|\phi| \le 1$,
by the classical divergence theorem on sets of finite perimeter, we have
\begin{align*}
   -\int_{\Omega} \tilde{u}\, \div\,\phi\,\dr x  =  -\int_{\Omega} u\, \div\,\phi\,\dr x
   = &- \lim_{j \rightarrow \infty} \int_{\Omega} w_j\, \div \phi\,\dr x\\
    =& \lim_{j \rightarrow \infty} \Big(\int_{\Omega} Dw_j \cdot \phi\,\dr x
       +\int_{\partial ^* \Omega} w_j\, \phi \cdot \nu_{\Omega}\,\dr x \Big)\\[1mm]
    \le & \limsup_{j \rightarrow \infty} \big(|Dw_j|(\Omega)+\|w_j\|_{\infty} P(\Omega)\big)< \infty.
\end{align*}
This completes the proof.
\end{proof}

In order to construct an example to answer the question negatively,
we consider the following natural class of Sobolev extension domains,
the so-called $M$--uniform domains.
Recall the following equivalent definition of $M$--uniform domains,
which was first introduced in \cite{GO} and \cite{Jo}.
\begin{definition}\label{uniformdomain}
Let $M>1$. We say that $\Omega $ is an $M$--uniform domain if, for any $x_1,x_2 \in \overline{\Omega}$,
there is a rectifiable curve $\gamma\,:\, [0,1] \rightarrow \overline{\Omega}$ with $\gamma(0)=x_1$ and $\gamma(1)=x_2$
such that
\begin{enumerate}
\item[\rm (i)]\label{Jones}
$\mathscr{H}^1(\gamma) \le M|x_1-x_2|$,

\smallskip
\item[\rm (ii)]
$d(\gamma(t), \partial \Omega) \ge \frac{1}{M} \min\{|\gamma(t)-x_1|,|\gamma(t)-x_2|\}$ $\quad$ for all $t \in [0,1]$.
\end{enumerate}
\end{definition}
It was proved in \cite{Jo} that, for an $M$--uniform domain,
constant $C$ in the definition of Sobolev extension domains depends only on $M$ and $n$.

Then the next example answers the question negatively.
\begin{example}\label{aha6}
Let $S$ be the classical Cantor ternary set defined in the closed interval $[0,1]$, by removing the middle thirds of the remaining interval in each step.
Let $\Omega=B_2\left((0,0)\right) \setminus (S\times S)$.
We observe that $\mathscr{H}^1(S \times S)=\infty$ and $|S\times S|=0$.
Since $\Omega$ is $\mathscr{H}^n$ equivalent to $B_2$, then $\Omega$ is a set of finite perimeter.
It is well known that $\Omega$ is a uniform domain so that, by Proposition {\rm \ref{haof}},
$\Omega$ is an extension domain for bounded BV functions satisfying \eqref{criteria}.
However, it is easy to check that $S \times S \subset \partial \Omega \cap \Omega^1$ so that
\begin{align*}
    \mathscr{H}^1(\partial \Omega \setminus \Omega^0) \ge \mathscr{H}^{1}(S\times S)=\infty.
\end{align*}
\end{example}

\smallskip
\subsection{Traces for bounded BV functions and Sobolev functions on extension domains}
For $u \in BV(\Omega)$, we can similarly define trace $Tu$ of $u$ in the sense of distributions:
\begin{align*}
    Tu(\phi):= \int_{\Omega} \phi \cdot \dr Du+\int_{\Omega} u\, \div\,\phi\, \dr x.
\end{align*}

From Example \ref{classical}, we know that even the trace of a bounded $BV$ vector field
is not necessarily concentrated on the reduced boundary of its domain.
However, if $\Omega$ satisfies
\begin{align}\label{cy3}
    \mathscr{H}^{n-1}(\Omega^1 \cap \partial \Omega)=0,
\end{align}
then the trace of a bounded $BV$ function $u$ is a function on $\partial^*\Omega$.

\begin{proposition}\label{bvtrace}
Let $u \in BV(\Omega) \cap L^{\infty}(\Omega)$.
If $\Omega$ is an open set of finite perimeter satisfying \eqref{cy3},
then there exists $u^* \in L^{\infty}(\partial^* \Omega)$ such that,
for any $\phi \in C_c^{\infty}(\mathbb{R}^n,\mathbb{R}^n)$,
the following integration by part formula holds{\rm :}
\begin{align}\label{cy4}
    \int_{\Omega} \phi \cdot \dr Du+\int_{\Omega} u\, \div\,\phi\,\dr x
    =-\int_{\partial^* \Omega} u^* \phi \cdot \nu_{\Omega}\, \dr\mathscr{H}^{n-1}.
\end{align}
In addition,
\begin{align}
\label{cy5}
\lim_{r \rightarrow 0} \frac{\int_{B_r(x) \cap \Omega}|u(y)-u^*(x)|\,\dr y}{r^n}=0 \qquad\,\, a.e. \,\,  x \in \partial^* \Omega.
\end{align}
\end{proposition}

\begin{proof}
Let \begin{align*}
    \tilde{u}(x)=\begin{cases}
    u(x) \quad &\mbox{for $x \in \Omega$},\\[1mm]
    0& \mbox{for $x \in \Omega^c$}.
    \end{cases}
\end{align*}
By Corollary \ref{exbv}, $\tilde{u} \in BV(\mathbb{R}^n)$. By \cite[Theorem 3.84]{afp},
\begin{align}
    D\tilde{u}=D\tilde{u} \res{\Omega^1}+u^* \nu_{\Omega} \mathscr{H}^{n-1}\res{\partial^* \Omega},
\end{align}
where $u^*$ satisfies \eqref{cy5}.
Using \eqref{cy3} and $Du \ll \mathscr{H}^{n-1}$, we have
\begin{align*}
    D\tilde{u}=&\,D\tilde{u} \res{\Omega}+u^* \nu_{\Omega} \mathscr{H}^{n-1}\res{\partial^* \Omega}
    =Du \res{\Omega}+u^* \nu_{\Omega} \mathscr{H}^{n-1}\res{\partial^* \Omega}.
\end{align*}
Then, for any $\phi \in C_c^1(\mathbb{R}^n;\mathbb{R}^n)$,
\begin{align*}
   \int_{\Omega} u\, \div\,\phi\,\dr x
   =\int_{\mathbb{R}^n} \tilde{u}\, \div\,\phi\, \dr x
   =- \int_{\mathbb{R}^n} \phi \cdot \dr D\tilde{u}=-\int_{\Omega} \phi \cdot \dr Du
    - \int_{\partial^* \Omega} u \phi \cdot \nu_{\Omega} \, \dr\mathscr{H}^{n-1}.
\end{align*}
This completes the proof.
\end{proof}
Even though Proposition \ref{bvtrace} above
is an immediate consequence of Corollary \ref{exbv}
and the standard results for BV functions,
to our knowledge, it has not known in the literature yet
since \eqref{neat-2} as a sufficient condition for extension domains
for bounded BV functions was unknown before.

The next remark says that the trace of $W^{1,1}$ functions is defined on the reduced boundary
of any Sobolev extension domain, which do not necessarily satisfy \eqref{cy3}.

\begin{remark}\label{sobtrace}
If $\Omega$ is a Sobolev extension domain and $u \in W^{1,1}(\Omega)$,
then there exists $u^*$ defined $\mathscr{H}^{n-1}$--a.e. $x \in \partial ^* \Omega$ such that,
for any $\phi \in C_c^{1}(\mathbb{R}^n,\mathbb{R}^n)$,
the following integration by parts formula holds{\rm :}
\begin{align}\label{cy4'}
    \int_{\Omega} \phi \cdot \dr Du +\int_{\Omega} u\, \div\,\phi\,\dr y
    =-\int_{\partial^* \Omega} u^* \phi \cdot \nu_{\Omega}\, \dr\mathscr{H}^{n-1}.
\end{align}
In addition,
\begin{align}\label{cy5'}
\lim_{r \rightarrow 0} \frac{\int_{B_r(x) \cap \Omega}|u(y)-u^*(x)|\, \dr y}{r^n}=0 \qquad \mathscr{H}^{n-1}-a.e. \,\,  x \in \partial^*\Omega.
\end{align}
\end{remark}

\begin{proof}
Since $\Omega$ is a Sobolev extension domain, then there exists $E{\rm :} \, W^{1,1}(\Omega) \rightarrow W^{1,1}(\mathbb{R}^n)$
such that $Eu=u$ in $\Omega$. Let
\begin{align*}
    \tilde{u}(x)=Eu \, \chi_{\Omega}.
\end{align*}
Again, by \cite[Theorem 3.84]{afp}, we have
\begin{align*}
    D\tilde{u}=&\,D(Eu) \res{\Omega^1}+u\, \nu_{\Omega} \mathscr{H}^{n-1}\res{\partial^* \Omega}\\
    =&\,D (Eu) \res{\Omega}+u\, \nu_{\Omega} \mathscr{H}^{n-1}\res{\partial^* \Omega}\\
    =&\, Du \res{\Omega}+u \,\nu_{\Omega} \mathscr{H}^{n-1}\res{\partial^* \Omega}.
\end{align*}
where we have used that $D(Eu) \ll\mathscr{H}^{n}$,  $\Omega$ is open,
and $Eu=u$ in $\Omega$.

Therefore, for any $\phi \in C_c^1(\mathbb{R}^n;\mathbb{R}^n)$,
\begin{align*}
   \int_{\Omega} u\, \div\,\phi\, \dr y
   =\int_{\mathbb{R}^n} \tilde{u}\, \div\,\phi\, \dr y
   =- \int_{\mathbb{R}^n} \phi \cdot \dr D\tilde{u}
   =-\int_{\Omega} \phi \cdot \dr Du- \int_{\partial^* \Omega} u\, \phi \cdot \nu_{\Omega} \, \dr\mathscr{H}^{n-1}.
\end{align*}
\end{proof}

\begin{remark}
It would be interesting to study further the relations between extension domains for bounded $BV$ functions,
bounded sets of finite perimeter, and bounded divergence-measure fields.
In particular, these include the questions
whether the following statements hold{\rm :}

\begin{enumerate}
\item[\rm (i)] A Sobolev extension domain of finite perimeter is still an extension domain for bounded divergence-measure fields{\rm .}
\item[\rm (ii)] Condition \eqref{neat-2} is also a necessary condition for an open set $\Omega$ to be an extension domain for bounded divergence-measure fields{\rm .}
\item[\rm (iii)] Condition \eqref{criteria} is sufficient for an open set $\Omega$ to be an extension domain for bounded BV functions{\rm .}
\item[\rm (iv)] Any extension domain for bounded BV functions is still an extension domain for bounded divergence-measure fields{\rm .}
\item[\rm (v)] The property, $g \in L^{\infty} \left( \partial  \Omega \setminus \Omega^0; \mathscr{H}^{n-1} \right)$, still holds, provided that
the open set $\Omega$ is required to be only an extension domain for bounded divergence-measure fields,
without assuming condition \eqref{neat-2}.
\end{enumerate}
\end{remark}

\begin{remark}
Let $\Omega$ be an open set satisfying \eqref{neat-2}.
We prescribe any $g \in L^{\infty}\left(\partial \Omega \setminus \Omega^0; \mathscr{H}^{n-1}\right)$.
An interesting question is whether there exists a vector field $\FF \in \mathcal{DM}^{\infty}(\Omega)$ such that
the normal trace of $\FF$
on $\partial \Omega$ corresponding to $g$.  Furthermore, for such $g$,  it is important to know whether problem \eqref{fanw}
without the upper $(n-1)$--Ahlfors condition can be solved.
Note that, with the upper $(n-1)$--Ahlfors regular condition imposed,
we have proved in Theorem {\rm 6.2} that the solvability of the divergence equation with prescribed $L^{\infty}$ boundary data
on $\partial \Omega \setminus \Omega^0$ is equivalent to the solvability of problem {\rm (6.4)} that
could be potentially easier to be solved owing to the nice structure of the reduced boundaries,
as indicated in Remark {\rm 6.3}.
On the other hand, the solvability question for problem {\rm (6.4)} in general domains is still open.
\end{remark}

\bigskip
\textbf{Acknowledgements}. $\,$
The research of
Gui-Qiang G. Chen was supported in part by
the UK
Engineering and Physical Sciences Research Council Award
EP/E035027/1 and
EP/L015811/1, and the Royal Society--Wolfson Research Merit Award (UK).
The research of Qinfeng Li was supported in part by the Purdue PRF Grant.
The research of Monica Torres was supported in part by the Simons Foundation Award No. 524190
and by the National Science Foundation Award \# 1813695.
The authors would like to thank Professor Changyou Wang for useful discussions on extension domains
for Sobolev spaces.

\end{document}